\documentclass[11pt,reqno,a4paper,english]{amsart}

\usepackage[utf8]{inputenc}  
\usepackage{dsfont}
\usepackage{xcolor}

\usepackage{enumerate}          
\usepackage{amssymb,amsmath,amsfonts,amsthm}            
\usepackage{pdfsync}
\usepackage{url}
\usepackage{hyperref}
\usepackage{anysize}
\usepackage{mathrsfs}
\usepackage{graphicx,float,color}
\usepackage{epsfig}
\usepackage{tikz}
\usepackage{caption}
\usepackage{cleveref}
\usepackage{mathtools}
\usepackage{esint}

\DeclareMathOperator{\Var}{Var}

\begin{document}
\theoremstyle{plain}
\newtheorem{theorem}{Theorem}[section]
\newtheorem{definition}[theorem]{Definition}
\newtheorem{proposition}[theorem]{Proposition}
\newtheorem{lemma}[theorem]{Lemma}

\newtheorem{corollary}[theorem]{Corollary}
\newtheorem{conjecture}[theorem]{Conjecture}
\newtheorem{remark}[theorem]{Remark}
\newtheorem{open}[theorem]{Open question}
\numberwithin{equation}{section}
\newtheorem*{claim}{Claim}
\newtheorem{example}[theorem]{Example}
\newtheorem{assumption}[theorem]{Assumption}

\newcommand{\Id}{\text{Id}}
\newcommand{\mnum}{A}
\newcommand{\mrat}{B}
\newcommand{\R}{\mathbb{R}}
\newcommand{\Z}{\mathbb{Z}}
\newcommand{\N}{\mathbb{N}}
\newcommand{\X}{\mathcal{X}}
\newcommand{\Y}{\mathcal{Y}}
\newcommand{\supp}{\text{supp}}
\newcommand{\red}{\color{red}}
\newcommand{\black}{\color{black}}
\newcommand{\A}{A}
\newcommand{\ppow}{\beta}
\newcommand{\gau}{\ppow}
\newcommand{\Cor}{\mathcal{C}}
\renewcommand{\H}{\mathcal{H}}
\newcommand{\B}{B}
\newcommand{\Ball}{\mathcal{B}}
\newcommand{\C}{C}
\newcommand{\De}{D}
\newcommand{\E}{E}
\newcommand{\ct}{5\B\sqrt{d}}
\newcommand{\foc}{{\rm foc}}
\newcommand{\dist}{{\rm dist}}

\newcommand{\osc}{\mathrm{osc}}
\newcommand{\TV}{\mathrm{TV}}
\newcommand{\sca}[2]{\langle #1 , #2\rangle}
\newcommand{\Kant}{\mathcal{K}}
\newcommand{\eps}{\varepsilon}
\newcommand{\dd}{\mathrm{d}}
\newcommand{\Class}{\mathcal{C}}
\newcommand{\Prob}{\mathcal{P}}
\newcommand{\D}{\mathrm{D}}
\newcommand{\IK}{\mathcal{I}}
\newcommand{\IG}{{\rm IG}}
\newcommand{\Rsp}{\mathbb{R}}
\renewcommand{\d}{\mathrm{d}}
\newcommand{\exponentiate}[1]{\exp\left(#1\right)}

\setcounter{tocdepth}{1}

\title[Stability of optimal transport maps on Riemannian manifolds]{Stability of optimal transport maps\\ on Riemannian manifolds}
\author{Jun Kitagawa}\address{Jun Kitagawa. Department of Mathematics, Michigan State University, 619 Red Cedar Road, East Lansing, MI 48824} \email{kitagawa@math.msu.edu}
 
\author{Cyril Letrouit}\address{Cyril Letrouit. Université Paris-Saclay, CNRS, Laboratoire de mathématiques d’Orsay, 91405, Orsay, France} \email{cyril.letrouit@universite-paris-saclay.fr}

\author{Quentin Mérigot}\address{Quentin Mérigot. Université Paris-Saclay, CNRS, Inria, Laboratoire de mathématiques d’Orsay, 91405, Orsay, France /
DMA, École normale supérieure, Université PSL, CNRS, 75005 Paris, France  / Institut universitaire de France (IUF)} \email{quentin.merigot@universite-paris-saclay.fr}


\date{\today}

\begin{abstract}
We prove quantitative bounds on the stability of optimal transport maps and Kantorovich potentials from a fixed source measure $\rho$ under variations of the target measure $\mu$, when the cost function is the squared  Riemannian distance on a Riemannian manifold. Previous works were restricted to subsets of Euclidean spaces, or made specific assumptions either on the manifold, or on the regularity of the transport maps. Our proof techniques combine  entropy-regularized optimal transport with spectral and integral-geometric techniques. As some of the arguments do not rely on the Riemannian structure, our work also paves the way towards understanding stability of optimal transport in more general geometric spaces.
\end{abstract}
\maketitle
\tableofcontents

\section{Introduction and main results}
\subsection{Motivation and setting} \label{s:motivationandsetting}

\subsubsection{Optimal transport}
The nearly 250 years old Monge transportation problem consists in finding the optimal way to transport mass from a given source to a given target probability measure, while minimizing an integrated cost. Mathematically, 
given two probability measures $\rho$ and $\mu$ defined on measurable spaces $\X$ and $\Y$ respectively, and a measurable \emph{cost function} $c:\X\times\Y\rightarrow \R$, Monge's problem is to find a measurable map $T:\X\rightarrow \Y$ with $T_{\#}\rho=\mu$ (i.e., $\mu(A)=\rho(T^{-1}(A))$ for any measurable $A\subset\Y$), minimizing the total transportation cost, meaning
\begin{equation}\label{e:mincost}
\int_\X c(x, T(x)) \dd\rho(x) = \min_{S_{\#}\rho=\mu} \int_\X c(x, S(x)) \dd\rho(x).
\end{equation}
Such a $T$ is called an \emph{optimal transport map}.

In \cite{brenier}, Brenier discovered that if both $\rho$ and $\mu$ are measures defined on $\R^d$ with the choice of quadratic cost $c(x,y)=|x-y|^2$, and if the source measure $\rho$ is absolutely continuous with respect to the Lebesgue measure, then there exists a $\rho$-a.e. unique solution $T$ of \eqref{e:mincost}. Moreover, $T$ is equal $\rho$-a.e. to the gradient of a convex function, and is the unique such map pushing  $\rho$ forward to $\mu$ given by the gradient of a convex function. 

Optimal transport has many applications, ranging from economics to image processing and machine learning;  
in various applications, one key issue is stability. A question of paramount importance to
practitioners is the following: if one perturbs the marginals $\rho$ and $\mu$, how much does the optimal transport map from $\rho$ to $\mu$ change? This question is also of strong theoretical interest, since stability of solutions, together with existence and uniqueness (here provided by Brenier's theorem), is one of the three components of a well-posed PDE problem. It has been known for a long time that the optimal transport map is stable in some weak sense with respect to perturbations of the marginals $\rho$ and $\mu$ (see e.g. \cite[Corollary 5.23]{villani2}), and the proof relies on non-quantitative compactness arguments. However, it is only in recent years that this problem has started being addressed quantitatively. 

Our paper is devoted to quantitative stability bounds for the optimal transport map (and the associated Kantorovich potentials, see below) when $\rho$ and $\mu$ are defined on Riemannian manifolds. Until now, the available techniques were not sufficient to address these quantitative stability questions in non-Euclidean spaces. Combining a regularization idea related to entropic optimal transport with some tools recently introduced in  \cite{delalandemerigot}, \cite{delalande}, \cite{letrouitmerigot}, we are able to bridge this gap in the present paper. Various ingredients in our approach (namely Sections~\ref{sec: kantorovich functional}, \ref{sec: I_rho}, and \ref{sec: Boman proof}) are robust enough to apply to more general metric spaces, and this paves the way toward understanding optimal transport stability in such settings.  In the next paragraph, we recall the basic framework of optimal transport in Riemannian manifolds.

\subsubsection{The optimal transport map in Riemannian manifolds}
Let $M$ be a smooth $d$-dimensional Riemannian manifold. For $\rho$ and $\mu$, two probability measures on $M$, we denote by $\mathcal{S}(\rho,\mu)$ the set of Borel maps $S$ which push $\rho$ forward to $\mu$, i.e., $S_{\#}\rho=\mu$. Monge's problem consists of finding the map $T: M\rightarrow M$ minimizing the transportation cost 
\begin{equation}\label{e:transportcost}
\mathscr{C}(S):=\int_M c(x,S(x)) d\rho(x)
\end{equation}
among all $S\in \mathcal{S}(\rho,\mu)$.
McCann's seminal paper \cite{mccann} solved the Monge problem on Riemannian manifolds, when $c$ is the Riemannian distance squared on $M$. McCann's results \cite[Theorems 8, 9, 13]{mccann} (see detailed statement in Theorem \ref{t:mccann} below) assert that if $\rho$ is compactly supported and absolutely continuous with respect to the volume measure on $M$, then there exists an optimal transport map which is essentially unique, and moreover it is of the form $T: x\mapsto \exp_x(-\nabla\phi(x))$ where $\exp$ denotes the exponential map, and $\phi$ is obtained as a $c$-transform, see \eqref{e:ctransform}. In the setting of our main result Theorem \ref{t:stabpote2}, if one imposes additionally that $\phi$ satisfies $\int_\X \phi \dd\rho=0$, then such a $\phi$ is unique. We call this the \emph{Kantorovich potential} from $\rho$ to $\mu$. 

\subsection{Main result}
Our main result is a bound which quantifies the stability of optimal transport maps and Kantorovich potentials with respect to variations of the target measure, on Riemannian manifolds. The source measure is fixed and assumed to be supported on a John domain $\X$, which we define immediately below. This assumption is satisfied in many cases of interest: in particular, in a Riemannian manifold $M$, any bounded (connected) Lipschitz domain is a John domain, and if $M$ is compact and connected, then $\X=M$ itself is a John domain of $M$.

\begin{definition}[John domains in metric spaces]\label{d:johnmetric}
A bounded open subset $\mathcal{X}$ of a metric space is called a John domain if there exist a distinguished point $x_0\in \mathcal{X}$ and a constant $\eta>0$ such that, for every $x\in\mathcal{X}$, there is a rectifiable curve $\gamma:[0,\ell(\gamma)]\rightarrow \mathcal{X}$ parametrized by the arclength (and whose length $\ell(\gamma)$ depends on $x$) such that $\gamma(0)=x$, $\gamma(\ell(\gamma))=x_0$, and for any $t\in[0,\ell(\gamma)]$
\begin{equation}\label{e:johncurve}
\dist(\gamma(t),\mathcal{X}^c)\geq \eta t,
\end{equation}
where $\mathcal{X}^c$ denotes the complement of $\mathcal{X}$.
\end{definition}
In other words, in a John domain it is possible to find a path from any point $x$ whose distance to the boundary grows at a linear rate, with a rate independent of $x$. John domains are automatically bounded and connected. John domains were introduced in \cite{john}, and the terminology was coined in \cite{martio}. Below, $W_1$ denotes the Wasserstein distance; since $W_1$ is the smallest of the $p$-Wasserstein distances, our result implies the same stability result for any higher order distance.

\begin{theorem}[Main result]\label{t:stabpote2}
Let $M$ be a smooth and connected Riemannian manifold, endowed with the quadratic cost $c(x,y)=\frac12 \dist(x,y)^2$ where $\dist$ denotes the Riemannian distance. Let $\mathcal{X}\subset M$ be a John domain and $\mathcal{Y}\subset M$ be compact, and let  $\rho$ be a probability measure absolutely continuous with respect to the Riemannian volume on $\mathcal{X}$, with density 
bounded from above and below by positive constants. Then there exists $C>0$ such that for any $\mu,\nu\in\mathcal{P}(\mathcal{Y})$,
\begin{equation}\label{e:varinrho0}
\|\phi_\mu-\phi_\nu\|_{L^2(\rho)}\leq C W_1(\mu,\nu)^{1/2}
\end{equation}
where $\phi_\mu$ denotes the (unique) Kantorovich potential from $\rho$ to $\mu$, and similarly $\phi_\nu$ denotes the unique Kantorovich potential from $\rho$ to $\nu$.
Moreover, if the boundary $\partial\X$ of $\X$ has finite $(d-1)$-dimensional Hausdorff measure, then
\begin{equation}\label{e:mapsstab}
\int_M \dist(T_\mu(x),T_\nu(x))^2d\rho(x)\leq CW_1(\mu,\nu)^{1/6}
\end{equation}
where $T_\mu$ and $T_\nu$ are the optimal transport maps from $\rho$ to $\mu$ and $\rho$ to $\nu$ respectively.
\end{theorem}

As already mentioned, when $M$ is closed and connected, the above theorem covers the case where $\X=M$:
\begin{corollary}[Closed manifolds]
Let $M$ be a smooth, connected and compact Riemannian manifold without boundary, $\Y\subset M$,  and $\rho$ be a probability measure absolutely continuous with respect to the Riemannian volume on $M$, with density bounded from above and below by positive constants. Then the conclusions of Theorem \ref{t:stabpote2} hold.
\end{corollary}
Theorem \ref{t:stabpote2} also covers the case where $\X$ is a bounded (connected) Lipschitz domain of a smooth Riemannian manifold, for instance a spherical cap, a bounded connected Lipschitz domain of $\R^d$ or of a hyperbolic space (see for instance \cite{alvarez} for applications of optimal transport on the hyperbolic space). In \cite[Theorem 1.9]{letrouitmerigot}, it has been shown that when $\rho$ is the uniform density on some appropriate non-John domain in $\R^d$, there exists no $C$, $\alpha$, $p >0$ such that the inequality $\|\phi_\mu-\phi_\nu\|_{L^2(\rho)}\leq CW_p(\mu,\nu)^\alpha$ holds for all $\mu,\nu$ supported in a given large compact set. In other words, in this case, Kantorovich potentials are extremely unstable. This shows the relevance of the John domain assumption in Theorem \ref{t:stabpote2}, at least regarding the stability of Kantorovich potentials.

Let us comment on the sharpness of the exponents in Theorem \ref{t:stabpote2}. Upper bounds with a $1/2$ exponent like in \eqref{e:varinrho0} have been derived in many places in the literature (see for instance \cite{delalandemerigot}, \cite{letrouitmerigot}, \cite{mischler}). We notice here for the first time that this $1/2$ exponent for Kantorovich potentials becomes sharp as $d\rightarrow +\infty$, see Remark \ref{r:sharp}. There are only few examples of probability measures $\rho$ for which a sharp exponent of stability for Kantorovich potentials has been derived, see \cite{letrouitmerigot}. In these examples, the density of $\rho$ is not bounded above and below by positive constants. Regarding stability of optimal transport maps, we do not know whether the exponent in \eqref{e:mapsstab} is sharp, it is only known that it cannot be larger than $1/2$ (see \cite{gigli}, \cite{merigot}).

Our results extend (and provide a new proof of) \cite[Theorem 1.7]{letrouitmerigot} by the last two authors, which corresponds to the particular case where $M=\R^d$. Our proof technique is different to that of \cite{letrouitmerigot}. The first part of our proof relies on a regularized version of the Kantorovich functional. In the second part we use gluing arguments for which we must show a so-called Boman chain condition, which is more intricate than in the Euclidean case. Note also that it should be possible to extend Theorem \ref{t:stabpote2} to cases where $\rho$ decays or blows-up polynomially, as in \cite[Section 1.3]{letrouitmerigot}, but for simplicity we shall not pursue this here.

As in \cite[Remark 4.2]{letrouitmerigot}, the connectedness of the support of $\rho$ is actually not necessary for the inequality \eqref{e:mapsstab} to hold: it is still true if $\X$ is a finite union of John domains whose boundary has finite $(d-1)$-dimensional Hausdorff measure, as may be seen with the same arguments as in \cite[Remark 4.2]{letrouitmerigot}.

We conjecture that our results are also true in more general metric measure spaces with synthetic curvature bounds -- as hinted for instance in Section \ref{s:convkantoreg} which is written in general metric spaces --, but again we shall not pursue this here.

\subsection{Related works}
Several recent works have established quantitative bounds on the stability of optimal transport maps and potentials in Euclidean spaces. In chronological order, \cite{gigli} initiated exploration of the subject by showing stability bounds around sufficiently regular transport maps. A few years later, \cite{berman} removed this regularity assumption and obtained quantitative stability bounds under the assumption that the source measure has a convex and compact support and is bounded above and below on this support (however, the stability exponents in \cite{berman} depend on the dimension). Later, a robust proof method was developed in \cite{merigot}, \cite{delalande} and \cite{delalandemerigot}, which removes the dimensional dependence in the results of \cite{berman}. Combined with gluing techniques which we shall also use in the present work, this proof strategy has been leveraged in \cite{letrouitmerigot} to cover much more general source measures, in particular (unbounded) log-concave measures, densities supported in John domains (with upper and lower bounds on the density), and compactly supported densities blowing-up polynomially at some points. 
In \cite{mischler}, the results of \cite{delalandemerigot} have been extended to optimal transport maps for the $p$-cost (with $p>1$). In \cite{divol}, the stability of entropic optimal transport has been considered, under minimal assumptions on the source and target measures; the bounds are tight, and blow-up as the entropic regularization parameter $\varepsilon$ tends to $0$. Finally, in a slightly different direction, quantitative stability of the associated pre-images in the semi-discrete case when the target measures remain supported on a fixed, finite set were treated in \cite{BansilKitagawa2020}.

Very few works in the literature deal with stability of optimal transport maps in non-Euclidean spaces, and existing works rely on quite specific assumptions. We are aware of \cite[Theorem 3.2]{ambrosio}, which proves stability bounds around sufficiently regular transport maps in Riemannian manifolds, in the spirit of \cite{gigli}. Also, in the recent paper \cite[Lemma 7, Appendix A.2]{sphere} extends the proof technique of \cite{merigot} to obtain a quantitative stability bound when $\rho$ is uniform on the sphere with exponent $1/9$. Our general bound \eqref{e:mapsstab} applies in particular to  this case, and improves the exponent to $1/6$. Finally, we mention that the results in \cite[Sections 5, 6, and 7]{KitagawaWarren12} and \cite[Sections 5 and 6]{JeongKitagawa25} show quantitative stability of transport maps in the uniform norm for more regular measures that are sufficiently close to the source measure (i.e., for transport maps close to the identity map), in the setting of cost function given by Euclidean distance squared restricted to the sphere or the boundary of a $C^1$ convex body, respectively; the exponents in both results depend on dimension and are likely non-sharp.

\subsection{Proof strategy and organization of the paper.}

In the present paper, we build on the general idea introduced in \cite{merigot} and \cite{delalandemerigot} of studying a particular functional $\mathcal{K}_\rho$ defined over the set of dual potentials $\psi\in\mathcal{C}^0(\Y)$ whose convexity/concavity properties translate into quantitative stability bounds for Kantorovich potentials. As in \cite{merigot} and \cite{mischler}, we actually study a family of ``regularized functionals" $\mathcal{K}_\rho^\varepsilon$, indexed by a parameter $\varepsilon>0$, and related to entropic (or regularized) optimal transport. We prove their strong concavity under the assumption that a strongly convex function exists on the manifold. This strong assumption is very restrictive; it is known to hold in sufficiently small balls, but strongly convex functions defined globally do not exist in general. As a consequence, the above arguments alone would only give a quantitative stability bound for $\rho$ supported in a sufficiently small ball of a Riemannian manifold. To overcome this difficulty, we then adapt the gluing techniques developed in \cite{letrouitmerigot} and manage to handle any measure $\rho$ bounded above and below on its support, assumed to be a John domain in a possibly non-compact Riemannian manifold $M$.

In Section \ref{s:convkantoreg} we define the Kantorovich regularized functional $\mathcal{K}_\rho^\varepsilon$, defined on $\mathcal{C}^0(\mathcal{Y})$ and depending on a ``regularization parameter" $\varepsilon$. We prove the strong concavity of $\mathcal{K}_\rho^\varepsilon$. This first part does not require $\rho$ to be absolutely continuous.

In Section \ref{s:existunique} we recall known facts about optimal transport in Riemannian manifolds. The strong concavity of $\mathcal{K}_\rho^\varepsilon$ is a key ingredient in Section \ref{s:stabbrpot} to prove the quantitative stability of Kantorovich potentials \eqref{e:varinrho0}. The other key ingredient is an adaptation of the gluing arguments of \cite{letrouitmerigot} to the setting of Riemannian manifolds. Then, in Section \ref{s:stabmapmanifo}, we deduce the stability of optimal transport maps \eqref{e:mapsstab}, relying on a 1 dimensional reverse Poincaré inequality and an integral geometric formula in Riemannian manifolds inspired by the Crofton formula.

\subsection{Notation.}
We denote by $\chi_S$ the characteristic function of a set $S$. In a metric space, the ball with center $x$ and radius $r$ is denoted by $B(x,r)$. Given a ball $Q$ and $r>0$, the ball with same center as $Q$ and radius multiplied by $r$ is denoted by $rQ$. The diameter of a set $S$ is denoted by ${\rm diam}(S)$.

In a Riemannian manifold $(M,g)$, we denote by $|v|_{g_x}$ the norm of $v\in T_xM$. The Riemannian distance between $x,y\in M$ is denoted by $\dist(x,y)$, and by extension $\dist(\cdot,\cdot)$ is also used to denote the distance in $M$ between a point and a set, or between two sets. The Ricci tensor is denoted by ${\rm Ric}$. The Hessian operator at $x\in M$ of a twice differentiable function $V:M\rightarrow \R$ is the
linear operator $\nabla^2V(x) : T_xM \rightarrow T_xM$ defined by the identity $\nabla^2V(x) v = \nabla_v(\nabla V)$, where $\nabla_v$ stands for the covariant derivative in the direction $v$. 

The $p$-Wasserstein distance between two probability measures on $M$ is denoted by $W_p(\cdot,\cdot)$, and the variance of a real-valued function $f$ on a space $\X$ equipped with a probability measure $\rho$ is 
$$
\Var_\rho(f):=\int_{\X}\left(f-\int_{\X}f\dd \rho\right)^2\dd\rho,
$$
The duality pairing between a measure $\mu$ and a continuous function $f$ is written $\sca{\mu}{f}$.

\subsection{Acknowledgments.} We would like to warmly thank Antoine Julia for his invaluable help in proving Proposition \ref{p:zelditch}. We are also grateful to  Pierre Pansu for a very interesting discussion related to this work. CL and QM acknowledge the support of the Agence nationale de la recherche, through the PEPR PDE-AI project (ANR-23-PEIA-0004). JK was partially supported by National Science Foundation grant DMS-2246606.

\section{Strong concavity of the Kantorovich regularized functionals}\label{s:convkantoreg}

In this section, we introduce and study a family of functionals indexed by a parameter $\varepsilon>0$, which we call the Kantorovich regularized functionals. Some of the computations presented in this section are similar to computations found in \cite{delalande}, \cite{chizat} and \cite{mischler}.

Let $(U,d)$ be a metric space and $Y$ a topological space. Additionally, assume $\rho$ and $\sigma$ are Borel probability measures on $U$ and $Y$ respectively. In the sequel we also make the assumption that $Y$ is compact, and the cost $c:U\times Y\rightarrow \mathbb{R}$ is bounded on $U\times Y$. Note that we make no further assumption on the measures $\rho$ and $\sigma$, including any kind of absolute continuity with respect to Hausdorff measure. We will make some additional assumptions in Section~\ref{sec: concavity of K}, and then also in Section~\ref{sec: potential stability proof}.

\subsection{The Kantorovich regularized functional and its first two derivatives}\label{sec: kantorovich functional}

Given $\eps>0$, and
$\psi\in\Class^0(Y)$, we define the $\eps$-regularized $c$-transform for $x\in U$ by
$$\psi^{c,\eps}(x) := -\eps \log\left(\int_Y \exponentiate{-\frac{c(x,y) - \psi(y)}{\eps}}\dd\sigma(y)\right), $$
and the $\eps$-regularized Kantorovich functional as 
$$ \Kant_\rho^\eps(\psi) := \int_U \psi^{c,\eps}(x)\dd\rho(x).$$
Since $c$ and $\psi$ are both bounded and $\rho$ and $\sigma$ are probability measures, we see $\psi^{c, \eps}(x)$ and $\Kant_\rho^\eps(\psi)$ are finite valued for all $x\in U$ and $\psi\in \Class^0(Y)$ respectively.

To each potential $\psi\in\Class^0(Y)$ and any point $x\in U$, we associate a probability density $\hat{\mu}_\eps^x[\psi]$ (with
respect to $\sigma$) and the corresponding probability measure
$\mu_\eps^x[\psi]$ on $Y$: 
\begin{equation}\label{e:muepschap}
\mu_\eps^x[\psi] := \hat{\mu}_\eps^x[\psi] \dd\sigma, \qquad \hat{\mu}_\eps^x[\psi] := \frac{\exponentiate{-\frac{c(x,y) - \psi(y)}{\eps}}}
{\int_Y \exponentiate{-\frac{c(x,z) - \psi(z)}{\eps}} \dd\sigma(z)}.
\end{equation}
We also define a probability measure $\mu_\eps[\psi]$ on $Y$ by integrating $\mu_\eps^x[\psi]$ with respect to $\rho(x)$:
\begin{equation}\label{e:mueps}
 \forall v \in\Class^0(Y), \quad \sca{\mu_\eps[\psi]}{v} := \int_U \sca{\mu_\eps^x[\psi]}{v} \dd\rho(x).
\end{equation}

\begin{remark}[Limit as  $\eps \to 0$]\label{r:convergence}
Note that if for a given $x\in U$, the minimum in the definition of the $c$-transform
\begin{equation}\label{e:mincpsi}
\psi^c(x)=\min_{y\in\Y} (c(x,y) - \psi(y))
\end{equation}
(see \eqref{s:stabbrpot} below) is attained at a unique point $y_\psi(x)$, then 
$$ \lim_{\eps\to 0}\mu_\eps^x[\psi] = \delta_{y_\psi(x)}. $$
This is the case for almost all $x$ when $c(x,y)=\frac12 \dist(x,y)^2$ is the quadratic cost on a Riemannian manifold (see \cite{mccann}, whose results are partly recalled in Theorem \ref{t:mccann}).
At least formally, $\mu_\eps[\psi]$ is the analogue of
$$ y_{\psi\#} \rho = \int_U \delta_{y_\psi(x)} \dd\rho(x).$$
Assuming that the support of $\sigma$ coincides with $\Y$, one can see that $\psi^{c,\eps}$ converges pointwise to $\psi^c$ as $\eps\to 0$ for any $\psi\in\Class^0(Y)$, and $\Kant_\rho^\eps(\psi) \rightarrow \Kant_\rho(\psi)$ as $\eps\to 0$.
\end{remark}

\begin{definition}[Gradient and Hessian]
Let $\mathcal{F}: \mathcal{C}^0(Y)\rightarrow \R$. If it exists, we define the gradient of
$\mathcal{F}$ at $\psi \in \mathcal{C}^0(Y)$ as the unique measure $\nabla \mathcal{F}(\psi)$ such that for all $v\in\Class^0(Y)$, $$\left.\frac{\dd}{\dd t} \mathcal{F}(\psi+tv)\right\vert_{t=0} =\sca{\nabla \mathcal{F}(\psi)}{v}.$$ Given $\psi\in\Class^0(Y)$ and $v\in\Class^0(Y)$, we denote the second-order derivative of $\mathcal{F}(\psi+tv)$ as $$ \sca{\D^2\mathcal{F}(\psi)v}{v} := \lim_{t\to 0} \sca{t^{-1}\left(\nabla \mathcal{F}(\psi+tv) - \nabla \mathcal{F}(\psi)\right)}{v}$$  when this limit exists.
\end{definition}
\begin{lemma}[Gradient of $\Kant_\rho^\eps$]\label{l:gradient}
The smoothed Kantorovich functional $\Kant_\rho^\eps$ is differentiable at any $\psi\in\Class^0(Y)$, and
 $$\nabla \Kant_\rho^\eps(\psi) =- \mu_\eps[\psi].$$
\end{lemma}

\begin{proof}
  Let $\psi_t = \psi+tv$, then for any $t$,
  \begin{align*}
      \left \lvert\frac{\dd}{\dd t}\exponentiate{-\frac{c(x,y)- \psi_t(y)}{\eps}}\right\rvert
      &\leq \frac{\lVert v\rVert_\infty}{\eps} \exponentiate{-\frac{c(x,y)-\psi(y)}{\eps}}\exponentiate{\frac{\lVert v\rVert_\infty}{\eps}}.
  \end{align*}
  Since $\lVert \psi\rVert_\infty$ and $\lVert v\rVert_\infty$ are finite, we can use dominated convergence to interchange differentiation and integration to compute
  \begin{equation}\label{eq:diffpsieps}
  \begin{aligned}
    \frac{\dd}{\dd t} \psi_t^{c,\eps}(x) &=- \eps \frac{\dd}{\dd t} \log\left(\int_Y  \exponentiate{-\frac{c(x,y)- \psi_t(y)}{\eps}} \dd\sigma(y)\right) \\
&=- \eps \frac{\frac{\dd}{\dd t}\int_Y \exponentiate{-\frac{c(x,y)- \psi_t(y)}{\eps}}  \dd\sigma(y)}{\int_Y \exponentiate{-\frac{c(x,y)- \psi_t(y)}{\eps}}  \dd\sigma(y)}\\
&=-  \frac{\int_Y v(y)\exponentiate{-\frac{c(x,y)- \psi_t(y)}{\eps}}  \dd\sigma(y)}{\int_Y \exponentiate{-\frac{c(x,y)- \psi_t(y)}{\eps}}  \dd\sigma(y)}.
  \end{aligned}
  \end{equation}
In particular this implies
  \begin{align}\label{eqn: psi derivative bounded}
    \Bigl\lvert\frac{\dd}{\dd t} \psi_t^{c,\eps}(x)\Bigr\rvert &=\frac{\lvert\int_Y v(y)\exponentiate{-\frac{c(x,y)- \psi_t(y)}{\eps}}  \dd\sigma(y)\rvert}{\int_Y \exponentiate{-\frac{c(x,y)- \psi_t(y)}{\eps}}  \dd\sigma(y)}
    \leq \lVert v\rVert_\infty.
  \end{align}
  Thus, by differentiating under the integral defining $\Kant_\rho^\eps(\psi)$ and setting $t=0$, we conclude that
\begin{equation*}
  \left.\frac{\dd}{\dd t} \Kant_\rho^\eps(\psi+tv)\right\vert_{t=0}=  \int_U \left(\left.\frac{\dd}{\dd t} \psi_t^{c,\eps}(x)\right\vert_{t=0}\right) \dd \rho(x)= - \int_U \sca{\mu_\eps^x[\psi]}{v} \dd\rho(x)= -\sca{\mu_\eps[\psi]}{v}.\qedhere
\end{equation*}
\end{proof}

\begin{lemma}\label{l:hesskanto}
 The regularized Kantorovich functional $\Kant_\rho^\eps$ is twice differentiable at
 any $\psi\in\Class^0(Y)$, and for any $v\in \Class^0(Y)$,
$$
\sca{\D^2\Kant_\rho^\eps(\psi)v}{v} = -\frac{1}{\eps} \int_{U}\Var_{\mu^x_\eps[\psi]}(v) \dd\rho(x).
$$
\end{lemma}

\begin{proof} Again let $\psi_t = \psi+tv$. Computing the derivative of the density $\hat{\mu}^x_\eps[\psi_t](y)$ for a fixed $(x, y)\in U\times Y$, we have
\begin{equation}\label{e:cxy}
- \frac{\dd}{\dd t} \hat{\mu}^x_\eps[\psi_t](y)=- \frac{\dd}{\dd t} \frac{\exponentiate{-\frac{c(x,y) - \psi_t(y)}{\eps}}}{\int_Y \exponentiate{-\frac{c(x,z) - \psi_t(z)}{\eps}} \dd\sigma(z)}  = - \frac{v(y)}{\eps} \hat{\mu}_\eps^x[\psi_t](y) + \frac{1}{\eps} \hat{\mu}^x_\eps[\psi_t](y) \sca{\mu^x_\eps[\psi_t]}{v},
\end{equation}
hence
\begin{align*}
   \left\lvert v(y) \frac{\dd}{\dd t} \hat{\mu}^x_\eps[\psi_t](y)\right\rvert
   &\leq \frac{2\lVert v\rVert_\infty^2}{\eps} \hat{\mu}_\eps^x[\psi_t](y)
   \leq \frac{2\lVert v\rVert_\infty^2 \exponentiate{\frac{2\lVert v\rVert_\infty}{\eps}}\exponentiate{-\frac{c(x, y)-\psi(y)}{\eps}}}{\eps\int_Y \exponentiate{-\frac{c(x, z)-\psi(z)}{\eps}}d\sigma(z)}\\
   &=\frac{2\lVert v\rVert_\infty^2 \exponentiate{\frac{2\lVert v\rVert_\infty}{\eps}}\hat \mu^x_\eps[\psi](y)}{\eps}.
\end{align*}
Since $\rho$ is a probability measure and $\lVert v\rVert_\infty<\infty$, the last expression above belongs to $L^1(\rho\otimes \sigma)$, hence we may again use dominated convergence to differentiate under the integral and use~\eqref{e:cxy} to obtain
\begin{align*}
  \left.\frac{\dd}{\dd t} \sca{\nabla \Kant_\rho^\eps(\psi_t)}{v}\right\vert_{t=0}
  &= -\left.\frac{\dd}{\dd t}\int_{U}  \int_Y v(y)\hat\mu_\eps^x[\psi_t](y)\dd\sigma(y)\dd\rho(x)\right\vert_{t=0}\\
  &=-\frac{1}{\eps}\int_U\int_Yv(y)\left(v(y) \hat{\mu}_\eps^x[\psi](y) -  \hat{\mu}^x_\eps[\psi](y) \sca{\mu^x_\eps[\psi]}{v}\right)\dd\sigma(y)\dd\rho(x)\\
&= -\frac{1}{\eps} \int_{U} \left(\sca{\mu^x_\eps[\psi]}{v^2} - \sca{\mu^x_\eps[\psi]}{v}^2 \right)\dd\rho(x) = -\frac{1}{\eps} \int_{U}\Var_{\mu^x_\eps[\psi]}(v) \dd\rho(x). \qedhere
\end{align*}
\end{proof}

\subsection{The functional $\mathcal{I}_\rho^\varepsilon$}\label{sec: I_rho}
To prove strong concavity of $\mathcal{K}_\rho^\varepsilon$, we first define and study another functional $\mathcal{I}_\rho^\varepsilon$, which is similar to a functional used by Cordero-Erausquin and Klartag \cite{cordero2015moment} to study moment measures. Given a probability measure $\rho$, we set 
$$ \IK_\rho^\eps(\psi) := \log\left(\int_U \exponentiate{\psi^{c,\eps}(x)} \dd \rho(x)\right).$$
%%%%%%%%%%%%%%%%%%%
Given $\psi \in\Class^0(Y)$, we consider the Gibbs density associated to 
$\psi^{c,\eps}$ denoted by $\hat \rho_\eps[\psi]$, and  the associated Gibbs measure $\rho_\eps[\psi]$, i.e. 
\begin{align*}
\hat \rho_\eps[\psi](x) := \frac{\exponentiate{\psi^{c,\eps}(x)}}{\int_U \exponentiate{\psi^{c,\eps}(z)}\dd\rho(z)},\qquad
\rho_\eps[\psi]:= \hat \rho_\eps[\psi]\dd\rho.
\end{align*}
Note that $\rho_\eps[\psi]$ is a probability measure, and again by boundedness of $c$ and $\psi$, and since $\rho$ is a probability measure, $\IK_\rho^\eps(\psi)$ and $\hat \rho_\eps[\psi]$ are finite.

\begin{proposition} \label{p:deriI} The functional $\mathcal{I}_\rho^\varepsilon$ is twice differentiable at any $\psi\in\mathcal{C}^0(Y)$, and its first and second derivatives in the direction $v\in\mathcal{C}^0(Y)$ are given by
\begin{align}
\sca{\nabla \IK_\rho^\eps(\psi)}{v} & =\int_{U} \sca{\mu_\eps^x[\psi]}{v}\dd\rho_\eps[\psi](x)\\
\sca{\D^2 \IK_\rho^\eps(\psi) v}{v} &= \Var_{\rho_\eps[\psi]}(x\mapsto \sca{\mu_\eps^x[\psi]}{v} )- \frac{1}{\varepsilon} \int_U \Var_{\mu_\eps^x[\psi]}(v) \dd\rho_\eps[\psi](x).
\end{align}
\end{proposition}

\begin{proof} Let $\psi_t = \psi+tv$. By~\eqref{eqn: psi derivative bounded}, for any $x\in U$ we have
\begin{align*}
    \left\lvert\frac{\dd}{\dd t}\exponentiate{\psi_t^{c,\eps}(x)}\right\rvert
    &\leq \lVert v\rVert_\infty \exponentiate{\psi_t^{c,\eps}(x)}
    \leq \lVert v\rVert_\infty \exponentiate{\lVert v\rVert_\infty}\left(\int_Y \exponentiate{-\frac{c(x, z)-\psi(z)}{\eps}}d\sigma(z)\right)^{-\eps}\\
    &=\lVert v\rVert_\infty \exponentiate{\lVert v\rVert_\infty}\exponentiate{\psi^{c,\eps}(x)}.
\end{align*}
Since the last expression belongs to $L^1(\rho)$, we may apply dominated convergence to differentiate under the integral and obtain
  \begin{align*}
    \frac{\dd}{\dd t} \exponentiate{\IK_\rho^\eps(\psi_t)} &=      \frac{\dd}{\dd t}\int_U \exponentiate{\psi_t^{c,\eps}(x)} \dd \rho(x) = \int_U \left(\frac{\dd}{\dd t} \psi_t^{c,\eps}(x)\right) \exponentiate{\psi_t^{c,\eps}(x)} \dd \rho(x)\\
    &= - \int_U \sca{\mu_\eps^x[\psi_t]} {v}\exponentiate{\psi_t^{c,\eps}(x)}\dd \rho(x),
  \end{align*}
  thus,
  \begin{align*}
      \sca{\nabla \IK_\rho^\eps(\psi_t)}{v} 
      &= -\exponentiate{-\IK_\rho^\eps(\psi_t)}\int_U \sca{\mu_\eps^x[\psi_t] }{v}\exponentiate{\psi_t^{c,\eps}(x)}\dd \rho(x) \\
  &=- \int_U \int_Y v(y)\mu_\eps^x[\psi_t](y) \hat\rho_\eps[\psi_t](x)\dd\sigma(y)\dd \rho(x).
  \end{align*}
 To calculate the second derivative, first note that
  \begin{align}
    \frac{\dd}{\dd t}\hat \rho_\eps[\psi_t](x) &= \frac{\dd}{\dd t}\frac{\exponentiate{\psi_t^{c,\eps}(x)}}{\int_U \exponentiate{\psi_t^{c,\eps}(z)}\dd \rho(z)}\nonumber\\
  %  &= \frac{\frac{\dd}{\dd t} \exponentiate{\psi_t^{c,\eps}(x)}}{\int_U \exponentiate{\psi_t^{c,\eps}(z)}\dd z}\rho(x) -  \frac{\exponentiate{\psi_t^{c,\eps}(x)} \frac{\dd}{\dd t} \int_U \exponentiate{\psi_t^{c,\eps}(z)}\dd z}{(\int_U \exponentiate{\psi^{c,\eps}(z)}\dd z)^2}\rho(x)\\
  %  &= \frac{ \left(\frac{\dd}{\dd t} \psi_t^{c,\eps}(x)\right) \exponentiate{\psi_t^{c,\eps}(x)}}{\int_U \exponentiate{\psi_t^{c,\eps}(z)}\dd \rho(z)}\rho(x) - \frac{\exponentiate{\psi_t^{c,\eps}(x)}  \int_U \left(\frac{\dd}{\dd t} \psi_t^{c,\eps}(z)\right)\exponentiate{\psi_t^{c,\eps}(z)}\dd \rho(z)}{(\int_U \exponentiate{\psi^{c,\eps}(z)}\dd \rho(z))^2}\rho(x)\\
    &=\hat\rho_\eps[\psi_t](x) \left(\int_U \sca{\mu_\eps^z[\psi_t]}{v} \dd\rho_\eps[\psi_t](z)
    -  \sca{\mu_\eps^x[\psi_t]}{v} \right)\label{e:ddtrhoeps}
  \end{align}
  where we used   \eqref{eq:diffpsieps}. Combining this with~\eqref{e:cxy}, we find
  \begin{align}
      \frac{\dd}{\dd t}\left(\hat{\mu}_\eps^x[\psi_t](y)\hat\rho_\eps[\psi_t](x)\right)
      &=\hat\rho_\eps[\psi_t](x)\left( \frac{v(y)}{\eps} \hat{\mu}_\eps^x[\psi_t](y) - \frac{1}{\eps} \hat{\mu}^x_\eps[\psi_t](y) \sca{\mu^x_\eps[\psi_t]}{v}\right.\notag\\
      &+\left.\int_U \sca{\mu_\eps^z[\psi_t]}{v} \dd\rho_\eps[\psi_t](z)
    -  \sca{\mu_\eps^x[\psi_t]}{v} \right)\label{eqn: rho hat derivative},
  \end{align}
  hence
  \begin{align*}
     \left\lvert \frac{\dd}{\dd t}\left(\hat{\mu}_\eps^x[\psi_t](y)\hat\rho_\eps[\psi_t](x)\right)\right\rvert
     &\leq 2\lVert v\rVert_\infty\hat\rho_\eps[\psi_t](x)\left( \frac{\hat{\mu}_\eps^x[\psi_t](y)}{\eps}  + 1 \right)\\
     &\leq 2\lVert v\rVert_\infty \exponentiate{2\lVert v\rVert_\infty}\hat\rho_\eps[\psi](x)\left( \frac{\exponentiate{\frac{2\lVert v\rVert_\infty}{\eps}}\hat \mu^x_\eps[\psi](y)}{\eps}  + 1 \right).
  \end{align*}
  Since this belongs to $L^1(\rho\otimes \sigma)$, we may again differentiate under the integral to find
  \begin{align*}
    \sca{\D^2 \IK_\rho^\eps(\psi) v}{v}= 
    \left.\frac{\dd}{\dd t} \sca{\nabla \IK_\rho^\eps(\psi_t)}{v} \right\vert_{t=0}
    &= -\int_U \int_Y v(y) \left(\left.\frac{\dd}{\dd t} \hat{\mu}_\eps^x[\psi_t](y)\right\vert_{t=0}\right)\dd\sigma(y) \dd\rho_\eps[\psi_t](x)\\
    &\qquad - \int_U \sca{\mu_\eps^x[\psi_t]}{v}\left(\left.\frac{\dd}{\dd t} \hat\rho_\eps[\psi_t](x)\right\vert_{t=0}\right)\dd \rho(x).
   % &= \frac{\dd}{\dd t} \int_U \int_Y  v(y) \hat{\mu}_\eps^x[\psi_t](y) \rho_\eps[\psi_t](x)\dd \sigma(y) \dd \rho(x) \\
   % &= \int_U \int_Y v(y) \left[\left(\frac{\dd}{\dd t} \hat{\mu}_\eps^x[\psi_t](y)\right) \rho_\eps[\psi_t](x) + \hat{\mu}_\eps^x[\psi_t](y)\left(\frac{\dd}{\dd t} \rho_\eps[\psi_t](x)\right)\right]\dd\sigma(y)\dd \rho(x)\\
  \end{align*}
  As in the proof of Lemma~\ref{l:hesskanto}, the first term of the sum above is equal to
  $$ - \frac{1}{\eps} \int_U \Var_{\mu_\eps^x[\psi]}(v) \dd\rho_\eps[\psi](x), $$ 
  while from~\eqref{e:ddtrhoeps}, the second term is 
  \begin{align*}
     \int_U \sca{\mu_\eps^x[\psi]}{v}^2 \dd\rho_\eps[\psi](x) - \left( \int_U \sca{\mu_\eps^x[\psi]}{v} \dd\rho_\eps[\psi](x)\right)^2
    &=\Var_{\rho_\eps[\psi]}(x\mapsto\sca{\mu_\eps^x[\psi]}{v} ),
  \end{align*}
  finishing the proof.
\end{proof}
\begin{remark}
It is possible to extend the results of this subsection to the case of noncompact $Y$, as long as $Y$ is locally compact. In this case, one should first assume $\psi\in \Class^0(Y)$ and the cost $c$ are such that $\exponentiate{-\frac{c(x, \cdot)-\psi}{\eps}}\in L^1(\sigma)$ and $\exp(\psi^{c, \eps})\in L^1(\rho)$ to ensure $\psi^{c, \eps}$ and $\mathcal{I}^\eps_\rho$ are well-defined. Additionally, in the proof of Lemma~\ref{l:gradient}, and Lemma~\ref{l:hesskanto} and Proposition~\ref{p:deriI}, the function $v$ should be chosen in the class $\Class_0(Y)$ of continuous functions decaying to zero at the boundary of $Y$. This will allow interchanging integral and derivative at various points in the above proof, and is sufficient to characterize the first and second derivatives of $\mathcal{K}^\eps_\rho$ and $\mathcal{I}^\eps_\rho$ as elements of the dual space $\Class_0(Y)^*$.
\end{remark}
\subsection{Strong concavity of Kantorovich's regularized functional}\label{sec: concavity of K}
We prove strong concavity estimates for $\Kant_\rho^\eps$, with a constant that does not deteriorate as $\eps\to 0$. The main idea is to deduce strong concavity of $\Kant_\rho^\eps$ from mere concavity of $\mathcal{I}_\rho^\varepsilon$. We show concavity of $\mathcal{I}_\rho^\varepsilon$ and strong concavity of $\Kant_\rho^\eps$ under some assumptions (Assumption \ref{a:assumptions}) which will be studied in detail in Section \ref{s:proofmain}.

From now on we fix a Riemannian manifold $(M, g)$. Compared to the setting described at the beginning of Section \ref{s:convkantoreg}, the set $U$ is now a subset of $M$, and $Y$ is still a compact topological space. The cost $c$ is assumed bounded on $U\times Y$, but it is not necessary for the moment to assume that it is the squared distance cost on $M$. As before, $\rho$ and $\sigma$ are Borel probability measures on $U$ and $Y$ respectively. We also make the following assumptions:

\begin{assumption}\label{a:assumptions} We assume:
\begin{itemize}
\item $U$ is a geodesically convex subset of $M$.
\item {\rm (Semi-concave cost).} 
There exists $\lambda\geq 0$ such that for all $y \in Y$ and all $x_0$, $x_1 \in U$ and any minimizing geodesic $(x_t)_{t\in [0,1]}$ connecting $x_0$ to $x_1$, one has
    \begin{equation} \label{eq:semi-convexity-distance}
    c(x_t, y) \geq (1 - t) c(x_0, y) + t c(x_1, y) - \frac{\lambda (1 - t) t \dist(x_0, x_1)^2}{2}. 
    \end{equation} 
\item {\rm (Lower bound on $\infty$-Bakry--Emery tensor).} 
There exists $V\in C^2(U)$ such that 
\begin{align}\label{eqn: Bakry-Emery}
    {\rm Hess\ } V+{\rm Ric}\geq \lambda.
\end{align}
\end{itemize}
\end{assumption}

\begin{remark}
The main result of the present paper, Theorem \ref{t:stabpote2}, is only stated for the squared quadratic cost $c(x,y)=\frac12 d(x,y)^2$, which is known to be semi-concave, i.e., to satisfy \eqref{eq:semi-convexity-distance}. More generally, for $K$-SC spaces in the sense of \cite{savare}, the squared distance cost is semi-concave by definition.  Semi-concavity is also true for instance for $c(x,y)=d(x,y)^p$ with $p\geq 2$.
Also, as noticed in \cite[Remark 3.2]{chizat}, whenever $U$ and $Y$ are both compact subsets of a smooth manifold and $c$ is smooth, $c$ is $\|c\|_{C^2(U\times Y)}$-semi-concave.
\end{remark}

An important consequence of the assumption~\eqref{eqn: Bakry-Emery} above is the following.
\begin{theorem}[Weighted Pr{\'e}kopa--Leindler inequality, {\cite[Theorem 1.4]{cordero}}]\label{thm: weighted PL}
 Suppose that $U$ is a geodesically convex subset of a Riemannian manifold $M$ such that~\eqref{eqn: Bakry-Emery} holds in $U$. Also suppose $s\in[0,1]$ and $f$, $g$, $h:U\rightarrow\R_+$ are such that 
$$
\forall x_0, x_1\in U,\ z\in Z_s(x_0,x_1), \quad h(z)\geq \exponentiate{-\lambda s(1-s)\dist(x_0,x_1)^2/2}f^{1-s}(x_0)g^s(x_1)
$$
where $Z_s(x_0,x_1)$ is the barycenter between $x_0$ and $x_1$ given by
$$
Z_s(x_0,x_1)=\{z\in U \mid \dist(x_0,z)=s\dist(x_0,x_1) \text{\  and\  } \dist(z,x_1)=(1-s)\dist(x_0,x_1)\}.
$$
Then 
\begin{equation*}
\int_U h \dd \rho^V\geq \Bigl(\int_U f \dd \rho^V\Bigr)^{1-s} \Bigl(\int_U g \dd \rho^V\Bigr)^s
\end{equation*}
where $\dd \rho^V=\frac{1}{Z} \exponentiate{-V} \dd{\rm vol}$ and $Z=\int_U \exponentiate{-V} \dd{\rm vol}$ is a normalizing constant.
\end{theorem}
\begin{remark}
Theorem \ref{thm: weighted PL} is stated in \cite[Theorem 1.4]{cordero} only for $U=M$, however the proof given in \cite{cordero} works for $U$ any geodesically convex subset of $M$, without any modification.
\end{remark}

\begin{proposition} \label{p:Iconcave} 
Under Assumption \ref{a:assumptions}, the functional $\IK_{\rho^V}^\eps$ is concave.
\end{proposition}

\begin{proof}[Proof of Proposition \ref{p:Iconcave}] 
Let $\psi_0$, $\psi_1 \in\Class^0(Y)$ and let $\psi_t = (1-t)\psi_0 + t\psi_1$. Given 
points  $y \in Y$ and $x_0$, $x_1 \in U$ and given a minimizing geodesic $(x_t)_{t\in [0,1]}$ one has according to \eqref{eq:semi-convexity-distance}
\begin{align*}
   c(x_t, y)-\psi_t(y)&\geq (1-t)(c(x_0, y)-\psi_0(y))+t(c(x_1, y)-\psi_1(y))-\frac{\lambda t(1-t)\dist(x_0, x_1)^2}{2}.
\end{align*}
Then
%\begin{align}
%   \exponentiate{-\frac{c(x_t, y)+\psi_t(y)}{\eps}}\sigma(y)\leq \exponentiate{ - (1-t)(c(x_0, y)+\psi_0(y))+t(c(x_1, y)+\psi_1(y))-\kappa t(1-t)\dd(x_0, x_1)^2} \sigma(y) 
%\end{align}
dividing the above inequality by $-\epsilon$, exponentiating, integrating over $Y$ against $\dd\sigma$, and finally  using convexity of the function  $v \in \Class^0(Y) \mapsto
\log(\int_Y \exponentiate{v(y)} \dd \sigma(y))$ we obtain
%\begin{align*}
%    \log\left(\int \exponentiate{-\frac{c(x_t, y)+\psi_t(y)}{\epsilon}}\dd\sigma(y)\right)&\leq \log\left(\left(\sum_{i=1}^N \exponentiate{-\frac{(1-t)(c(x_0, y_i)+(\psi_0)_i)}{\ilon}-\frac{t(c(x_1, y_i)+(\psi_1)_i)}{\epsilon}}\right)\exponentiate{\frac{\tilde Kt(1-t)}{2\epsilon}d_g(x_0, x_1)^2}\right)\\
%    &\leq \log\left(\left(\sum_{i=1}^N \exponentiate{\frac{-(1-t)(c(x_0, y_i)+(\psi_0)_i)}{\epsilon}}\right)\left(\sum_{i=1}^N \exponentiate{\frac{-t(c(x_1, y_i)+(\psi_1)_i)}{\epsilon}}\right)\exponentiate{\frac{\tilde Kt(1-t)}{2\epsilon}d_g(x_0, x_1)^2}\right)\\
%    &=\log\left(\sum_{i=1}^N \exponentiate{\frac{-c(x_0, y_i)+(\psi_0)_i}{\epsilon}}\right)^{1-t}+\log\left(\sum_{i=1}^N \exponentiate{\frac{-c(x_1, y_i)+(\psi_1)_i}{\epsilon}}\right)^t+\frac{\tilde Kt(1-t)}{2\epsilon}d_g(x_0, x_1)^2,
%\end{align*}
%or equivalently
\begin{align*}
    \psi_t^{c, \epsilon}(x_t) \geq (1-t) \psi_0^{c, \epsilon}(x_0) + t \psi_1^{c, \epsilon}(x_1) - \frac{\lambda t(1-t) \dist(x_0, x_1)^2}{2}.
\end{align*}
Exponentiating yields 
\begin{align*}
    \exponentiate{\psi_t^{c, \epsilon}(x_t)}&\geq (\exponentiate{\psi_0^{c, \epsilon}(x_0)})^{1-t} (\exponentiate{\psi_1^{c, \epsilon}(x_1)})^{t}\exponentiate{- \frac{\lambda t(1-t) \dist(x_0, x_1)^2}{2}}.
\end{align*}    
By \eqref{eqn: Bakry-Emery} we may apply Theorem~\ref{thm: weighted PL}, which yields
\begin{align*}
     \int_U \exponentiate{\psi_t^{c, \epsilon}(x)} \d \rho^V(x)&\geq\left(\int_U\exponentiate{\psi_0^{c, \epsilon}(x)}\d \rho^V(x)\right)^{1-t}\left(\int_U\exponentiate{\psi_1^{c, \epsilon}(x)}\d \rho^V(x)\right)^{t},
\end{align*}
showing the concavity of $\IK_{\rho^V}^\eps$. 
\end{proof}
We finally deduce from the concavity of $\IK_{\rho^V}^\eps$ the strong concavity of $\Kant_{\rho^V}^\eps$.
\begin{proposition}\label{p:D2K}
Suppose that Assumption \ref{a:assumptions} holds and the cost $c$ has bounded oscillations, i.e. 
$$
\|c\|_\osc:=\sup_{x,x'\in U, \ y\in Y} (c(x,y)-c(x',y))<+\infty.
$$
Then
\begin{equation}\label{e:Keps}
\langle D^2\mathcal{K}^\varepsilon_{\rho^V}[\psi_t]v,v\rangle\leq -C_0^{-2}  \Var_{\rho^V}(x\mapsto \sca{\mu_\eps^x[\psi_t]}{v} )
\end{equation}
for $C_0:=\exp(\|c\|_{\osc})$. 
\end{proposition}
\begin{proof}[Proof of Proposition \ref{p:D2K}]
Since $\mathcal{I}_{\rho^V}^\varepsilon$  is concave according to Proposition \ref{p:Iconcave}, we get from Proposition \ref{p:deriI} that
\begin{equation}\label{e:but}
\frac{1}{\varepsilon} \int_U \Var_{\mu_\eps^x[\psi]}(v) \dd\rho^V_\eps[\psi](x) \geq  \Var_{\rho_\eps^V[\psi]}(x\mapsto \sca{\mu_\eps^x[\psi]}{v}).
\end{equation}
We observe that for any $x$, $x'$
$$
\int_Y \exponentiate{-\frac{c(x,y)-\psi(y)}{\varepsilon}}d\sigma(y)\leq \exponentiate{\frac{\|c\|_\osc}{\varepsilon}} \int_Y \exponentiate{-\frac{c(x',y)-\psi(y)}{\varepsilon}}d\sigma(y)
$$
hence $\sup \psi^{c,\varepsilon} - \inf \psi^{c,\varepsilon}\leq \|c\|_\osc$.
Therefore, setting $C_0=\exp(\lVert c\rVert_{\osc})$, we have
$$
C_0^{-1}\leq \hat{\rho}^V_\varepsilon[\psi](x)\leq C_0.
$$
The expression on the left side of \eqref{e:but} is therefore bounded from above by the quantity $\frac{C_0}{\eps} \int_U\Var_{\mu^x_\eps[\psi]}(v) \dd\rho^V(x)=-C_0\langle D^2\mathcal{K}^\varepsilon_{\rho^V}[\psi]v,v\rangle$, where the last equality is due to Lemma \ref{l:hesskanto}. For the same reason, the right-hand side in \eqref{e:but} is bounded from below by $ C_0^{-1}\Var_{\rho^V}(x\mapsto \sca{\mu_\eps^x[\psi]}{v} )$. All in all, we get \eqref{e:Keps}.
\end{proof}

\section{Stability of Kantorovich potentials}\label{s:proofmain}
In this section, we leverage the strong concavity of Kantorovich's regularized functional to show \eqref{e:varinrho0} in Theorem \ref{t:stabpote2}, i.e.,  stability of the Kantorovich potentials.

\subsection{Existence and uniqueness of Kantorovich potentials and optimal transport maps}\label{s:existunique}

We start by recalling the following result, due to McCann, which shows the existence of Kantorovich potentials and optimal transport maps in Riemannian manifolds.
\begin{theorem}\cite[Theorem 10.41]{villani2}, \cite[Theorem 13]{mccann} \label{t:mccann}
Let $M$ be a Riemannian manifold endowed with the quadratic cost $c(x,y)=\frac12 d(x,y)^2$. Let $\rho$, $\mu$ be probability measures compactly supported in $M$, with $\rho$ absolutely continuous. Assume that the optimal transport cost from $\rho$ to $\mu$ is finite. Then there is a $\rho$-a.e. unique solution of the Monge problem \eqref{e:mincost} between $\rho$ and $\mu$, and it can be written as 
\begin{equation}\label{e:t}
T(x)=\exp_x\left(-\nabla\phi(x)\right)
\end{equation}
where $\phi=\psi^c$ defined in \eqref{e:ctransform}, for some $\psi:\Y\rightarrow \R\cup\{+\infty\}$.
\end{theorem}

Uniqueness of $\phi$, up to a constant in each connected component of the support of $\rho$, follows from \cite[Remark 10.30]{villani2}. It is also a byproduct of our stability inequality for Kantorovich potentials \eqref{e:varinrho0}, see footnote on page~\pageref{foot}.

\begin{definition}
Given $\rho$ and $\mu$ two Borel probability measures on $M$, with $\rho$ absolutely continuous, we call the $\rho$-a.e. unique map $T\in \mathcal{S}(\rho,\mu)$ minimizing $\mathscr{C}$ (defined in~\eqref{e:transportcost}) the \emph{optimal transport map}. Also, if the support of $\rho$ is connected, the \emph{Kantorovich potential} is the unique Lipschitz function $\phi: M\rightarrow \R$ such that $T(x)=\exp_x(-\nabla\phi(x))$ for $\rho$-a.e. $x$ and $\int_M \phi(x)d\rho(x)=0$.
\end{definition}

\subsection{Proof of \eqref{e:varinrho0}}\label{s:stabbrpot}
Recall that the $c$-transform
\begin{equation}\label{e:ctransform}
\psi^c:x\mapsto \inf_{y\in\mathcal{Y}} (c(x,y)-\psi(y))
\end{equation}
on a metric space of finite diameter is either identically infinite ($\psi=\pm\infty$), or Lipschitz continuous in $M$ (see \cite[Lemma 2]{mccann}).

The first part of Theorem \ref{t:stabpote2} is an almost direct consequence of the following proposition, whose proof is given in Section \ref{s:proofprop}.
\begin{proposition}\label{p:stabpote}
Let $\mathcal{X}$ be a John domain in a smooth and complete Riemannian manifold $(M, g)$. Let $\rho=\hat\rho{\rm dvol}$ be a probability measure on $\mathcal{X}$ satisying $0<m_\rho\leq \hat\rho\leq M_\rho<+\infty$, let $c(x,y)=\frac12 \dist(x,y)^2$, and let $\Y\subset M$ be a compact set. Then there exists $C>0$ such that for any $\psi_0$, $\psi_1\in C^0(\Y)$,
\begin{equation}\label{e:varinrho01}
\Var_\rho(\psi_1^{c}-\psi_0^{c})\leq C  \sca{(y_{\psi_1})_{\#}\rho-(y_{\psi_0})_{\#}\rho}{\psi_1-\psi_0}
\end{equation}
where $y_\psi$ is defined $\rho$-a.e. by
\begin{equation}\label{e:ypsi}
y_\psi:x\mapsto \exp_x(-\nabla \psi^c(x)).
\end{equation}
\end{proposition}
Let us explain how to conclude the proof of the first part of Theorem \ref{t:stabpote2}.
\begin{proof}[Proof of \eqref{e:varinrho0}]  
We specialize Proposition \ref{p:stabpote} to the case where $\psi_0$ (resp. $\psi_1$) is \label{foot}a\footnote{This is actually unique up to constants, but we do not need to know this to prove \eqref{e:varinrho0}; in fact uniqueness may be obtained as a consequence of \eqref{e:varinrho0}, taking $\mu=\nu$.} dual Kantorovich potential associated to the optimal transport problem from $\rho$ to $\mu$ (resp. from $\rho$ to $\nu$). In particular, $(y_{\psi_0})_{\#}\rho=\mu$ and $(y_{\psi_1})_{\#}\rho=\nu$. Applying the Kantorovich-Rubinstein duality formula (see \cite[Theorem 1.14]{villani}) to the right-hand side, since $\psi_0$ and $\psi_1$ are ${\rm diam}(\mathcal{X})$-Lipschitz, we deduce 
$$
\Var_\rho(\phi_\mu-\phi_\nu)\leq C{\rm diam}(\X)W_1(\mu,\nu).
$$
This concludes the proof.
\end{proof}

\begin{remark}\label{r:sharp}
Let us show that the exponent $1/2$ in \eqref{e:varinrho0} becomes asymptotically sharp as $d\rightarrow +\infty$.
Denote by $\omega_d$ the Euclidean volume of the unit ball $B_d(0,1)$ of $\R^d$ and by $\sigma_{d-1}$ the Euclidean area of the unit sphere $\mathbb{S}^{d-1}\subset\R^d$. Let
$$
\rho_d(x)=\frac{1}{\omega_d}\chi_{B_d(0,1)}
$$
be the uniform probability density on the unit ball of $\R^d$. Consider for any $\varepsilon\in(0,1)$ the radial and convex functions
$$
\phi_\varepsilon^{(1)}(x)=|x|, \qquad \phi_\varepsilon^{(2)}(x)=\max(|x|,\varepsilon).
$$
Then
$$
\int_{B_d(0,1)} \left(\phi_\varepsilon^{(2)}-\phi_\varepsilon^{(1)}\right)\dd\rho=\frac{\sigma_{d-1}}{\omega_d}\int_0^\varepsilon r^{d-1}(\varepsilon-r)\dd r= \frac{\varepsilon^{d+1}}{d+1}
$$
and 
$$
\int_{B_d(0,1)} \left(\phi_\varepsilon^{(2)}-\phi_\varepsilon^{(1)}\right)^2\dd\rho=\frac{\sigma_{d-1}}{\omega_d}\int_0^\varepsilon r^{d-1}(\varepsilon-r)^2\dd r=\frac{2\varepsilon^{d+2}}{(d+1)(d+2)}.
$$
Hence, 
\begin{equation}\label{e:varphi1phi2eps}
\Var(\phi_\varepsilon^{(2)}-\phi_\varepsilon^{(1)})^{1/2}\sim c_d \varepsilon^{(d+2)/2}
\end{equation}
as $\varepsilon\rightarrow 0$, with $c_d=(2/(d+1)(d+2))^{1/2}$.
Finally, denoting by $\delta_{\mathbb{S}^{d-1}}$ the uniform probability measure on $\mathbb{S}^{d-1}$,
$$
(\nabla\phi_\varepsilon^{(1)})_\#\rho_d=\delta_{\mathbb{S}^{d-1}}, \qquad (\nabla\phi_\varepsilon^{(2)})_\#\rho_d= (1-\varepsilon^{d})\delta_{\mathbb{S}^{d-1}}+\varepsilon^d\delta_0
$$
hence 
$$
W_1((\nabla\phi_\varepsilon^{(1)})_\#\rho_d,(\nabla\phi_\varepsilon^{(2)})_\#\rho_d)=\varepsilon^d.
$$
Comparing with \eqref{e:varphi1phi2eps} we obtain a stability exponent $\alpha_d=(d+2)/2d$, which tends to $1/2$ as $d\rightarrow +\infty$.
\end{remark}

\subsection{Proof of Proposition \ref{p:stabpote}} \label{s:proofprop}
We now conclude Section \ref{s:proofmain} by completing the missing step in the proof of \eqref{e:varinrho0}, namely we prove Proposition \ref{p:stabpote}.

\subsubsection{Boman chains in metric spaces with doubling measures.} \label{sec: Boman proof}

In this section, we prove that any John domain $\mathcal{X}$ of a metric space endowed with a doubling measure admits a decomposition into balls with some specific properties listed below. In the Euclidean case, a simple version of this decomposition, deduced from the Whitney decomposition, is used in \cite[Section 3.1]{letrouitmerigot} (with cubes instead of balls, but this plays no role here). This decomposition is used to prove that the Boman chain condition, first introduced in \cite{boman} in a Euclidean framework, holds. Here, we show how to generalize these arguments to any John domain of a metric space endowed with a doubling measure $\rho$. 

In the sequel, given a ball $Q$ and $r>0$, the ball with same center as $Q$ and radius multiplied by $r$ is denoted by $rQ$. Recall that a measure $\rho$ on a metric space $S$ is called \emph{doubling} if there exists $C>0$ such that for any ball $Q\subset S$, $\rho(2Q)\leq C\rho(Q)$. 

\begin{definition}[Boman chain condition]\label{d:bomanmanifolds}
Let $\A$, $\B$, $\C>1$. A probability measure $\rho$ on an open set $\X$ of a metric space satisfies the Boman chain condition with parameters $\A,\B,\C\in\R$ if there exists a covering $\mathcal{F}$ of $\X$ by open balls $Q\in\mathcal{F}$ such that:
\begin{itemize}
\item For any $x\in\R^d$,
\begin{equation}\label{e:A}
\sum_{Q\in\mathcal{F}}\chi_{2 Q}(x)\leq \A\chi_\X(x).
\end{equation}
\item For some fixed ball $Q_0$ in $\mathcal{F}$, called the central ball, and for every $Q\in\mathcal{F}$, there exists a chain $Q_0,Q_1,\ldots,Q_N=Q$ of distinct balls from $\mathcal{F}$ such that for any $j\in \{0,\ldots,N-1\}$,
\begin{equation}\label{e:BJohn}
Q\subset \B Q_j.
\end{equation}
\item Consecutive balls of the above chain overlap quantitatively: for any $j\in \{0,\ldots,N-1\}$,
\begin{equation}\label{e:conseccubeoverlap}
\rho(Q_j\cap Q_{j+1})\geq \C^{-1}\max(\rho(Q_j),\rho(Q_{j+1})).
\end{equation}
\end{itemize}
\end{definition}
\begin{remark}
    Note that in Boman's original paper, he includes the following condition along with \eqref{e:conseccubeoverlap} (see \cite[Lemma 2.1 (2.3) (a)]{boman}):

    Consecutive balls of the chain have comparable size: for any $j\in \{0,\ldots,N-1\}$,
\begin{equation}\label{e:conseccubessize}
\C^{-1}\leq \frac{\rho(Q_j)}{\rho(Q_{j+1})}\leq \C.
\end{equation}
We observe here that \eqref{e:conseccubessize} is a direct consequence of \eqref{e:conseccubeoverlap}: indeed, assuming that \eqref{e:conseccubeoverlap} holds, we have
$$
\frac{\rho(Q_j)}{\rho(Q_{j+1})}\geq \frac{\rho(Q_j\cap Q_{j+1})}{\rho(Q_{j+1})}\geq C^{-1}\frac{\max(\rho(Q_j),\rho(Q_{j+1}))}{\rho(Q_{j+1})}\geq C^{-1}
$$
and similarly, $\frac{\rho(Q_{j+1})}{\rho(Q_j)}\geq C^{-1}$. 
\end{remark}
\begin{proposition}[Boman chain condition in John domains of metric spaces]\label{prop: boman chain construction}
Let $\rho$ be a probability measure which is the restriction to a John domain $\mathcal{X}$ of a doubling measure on a metric space. Then $\rho$ satisfies the Boman chain condition for some $\A$, $\B$, $\C>1$. Moreover, for any $R>0$, the covering $\mathcal{F}$ may be chosen in a way such that all its elements $Q\in\mathcal{F}$ have radius at most $R$.
\end{proposition}
\begin{proof}
We slightly modify the construction of \cite[Proposition 4.1.15]{heinonen}. Let $R>0$ and assume without loss of generality that $R\leq {\rm diam}(M)$. For $x\in\mathcal{X}$, let 
$$
\delta(x)=\min(\dist(x, \mathcal{X}^c),R),
$$
and for $k\in\Z$, let
$$
\mathcal{F}_k:=\left\{B(x,\delta(x)/100) \mid x\in \mathcal{X}\ {\rm with }\ 2^{k-1}<\delta(x)\leq 2^k\right\}.
$$
By the (infinite) Vitali covering lemma, for any $k\in\Z$, we can pick a countable pairwise disjoint subfamily $\mathcal{G}_k\subset\mathcal{F}_k$ such that
\begin{equation}\label{e:FkGk}
\text{for all }  Q\in \mathcal{F}_k, \text{ there exists }  Q'\in\mathcal{G}_k \text{  such that  } Q\subset 5Q',
\end{equation}
and as a consequence
$$
\bigcup_{Q\in\mathcal{F}_k}Q\subset \bigcup_{Q\in \mathcal{G}_k} 5Q.
$$
We consider the (countable) collection 
\begin{align*}
    \mathcal{F}:=\bigcup_{k\in \Z}\left\{5Q\mid Q\in\mathcal{G}_k\right\}.
\end{align*} 

Let us check that $\mathcal{F}$ verifies \eqref{e:A}. For some $p\in \N$, assume there is a point $x\in \bigcap_{i=1}^p B(x_i,\frac{1}{20} \delta(x_i))$ where $B(x_i,\frac{1}{20} \delta(x_i))\in\mathcal{F}$ for $1\leq i\leq p$. Without loss of generality, assume $\delta(x_1)= 
\max_{1\leq i\leq p}\delta(x_i)$, then we can show that
\begin{equation}\label{e:deltaxideltax1}
\delta(x_i)\geq \frac{9}{11} \delta(x_1).
\end{equation}
Indeed, since $B_1\cap B_i\neq \emptyset$, 
$$
\dist(x_i,\X^c)\geq \dist(x_1,\X^c)-\frac{1}{10}\delta(x_i)-\frac{1}{10}\delta(x_1)
$$
and due to the definition of $\delta(x)$ as a minimum, we may replace $\dist(x_1,\X^c)$ and $\dist(x_i,\X^c)$ respectively by $\delta(x_1)$ and $\delta(x_i)$, which yields \eqref{e:deltaxideltax1}. Let $k_1\in\Z$ be such that $B(x_1,\frac{1}{100}\delta(x_1))\in \mathcal{F}_{k_1}$, it then follows from \eqref{e:deltaxideltax1} and since $\delta(x_1)\geq \delta(x_i)$ that 
$$
\forall i=1,\ldots,p, \qquad B\Bigl(x_i,\frac{1}{100}\delta(x_i)\Bigr)\in\mathcal{F}_{k_1-1}\cup \mathcal{F}_{k_1}.
$$
We deduce that either $\mathcal{F}_{k_1}$ or $\mathcal{F}_{k_1-1}$ contains at least $\ell\geq p/2$ balls $B(x_i,\frac{1}{100}\delta(x_i))$, whose centers we denote by $x_{i_1},\ldots,x_{i_\ell}$. Then for $1\leq j_1, j_2\leq \ell$, we see
$$
\dist(x_{j_1},x_{j_2})\geq \frac{1}{50}\min\{\delta(x_{j_1}),\delta(x_{j_2})\}\geq \frac{1}{70}\delta(x_1)
$$
where the first inequality comes from the fact that each $\mathcal{G}_k$ is a family of \emph{disjoint} balls, and the second one from \eqref{e:deltaxideltax1}. 

 Combining the above, we find that the ball $B(x_1,\frac{3}{10} \delta(x_1))$ contains the $\frac{1}{70}\delta(x_1)$-separated set $\{x_{i_1},\ldots,x_{i_\ell}\}$ and 
$$
B(x_{i_j},\frac{1}{10}\delta(x_{i_j}))\subset B(x_1,\frac{3}{10}\delta(x_1)),\qquad \forall 1\leq j\leq \ell,
$$
due to the fact that $\delta(x_1)\geq \delta(x_i)$. Then using \eqref{e:deltaxideltax1} together with the fact that $\dist(x_1,x_i)\leq \frac{1}{10} (\delta(x_1)+\delta(x_i))$, we get that $B(x_1,\frac{3}{10}\delta(x_1))\subset B(x_i,\frac35\delta(x_i))$ for any $i=1,\ldots,p$. Therefore using that $\rho$ is a doubling measure,
$$
\rho(B(x_1,\frac{3}{10}\delta(x_1)))\geq \sum_{j=1}^\ell \rho(B(x_{i_j},\frac{1}{140}\delta(x_{i_j}))) \geq C\sum_{j=1}^\ell \rho(B(x_{i_j},\frac35\delta(x_{i_j})))\geq C\ell \rho(B(x_1,\frac{3}{10}\delta(x_1))).
$$
As a result $\ell$, and thus $p$, are bounded above by a constant $\A$ which only depends on $\rho$.

Now let $x_0$ be from Definition \ref{d:johnmetric}, and  choose $Q_0\in\mathcal{F}$ satisfying $B(x_0,\delta(x_0)/100)\subset Q_0$, which exists according to \eqref{e:FkGk}. For an arbitrary $Q\in\mathcal{F}$, let us construct the chain from $Q_0$ to $Q$ satisfying the properties of Definition \ref{d:bomanmanifolds}. For this we adapt the proof of \cite[Lemma 2.1]{boman} which handles the Euclidean case. Let $x$ be the center of $Q$, and take according to Definition \ref{d:johnmetric} a curve $\gamma$ from $x$ to $x_0$ satisfying \eqref{e:johncurve}. Without loss of generality we also assume that $\eta\leq 1$. We claim that for all $z\in\gamma$
\begin{equation}\label{e:ratiodisteta}
\frac{\dist(z,\mathcal{X}^c)}{\dist(x,\mathcal{X}^c)}\geq \frac{\eta}{2}.
\end{equation}
Assume \eqref{e:ratiodisteta} does not hold, hence $\dist(z,\X^c)< \frac{\eta}{2} \dist(x,\X^c)$ for some $z=\gamma(t)$. Then combining with the triangle inequality,
\begin{align*}
   \dist(z,\X^c)< \frac{\eta}{2}\dist(x,\X^c)\leq \frac{\eta}{2}(\dist(x,z)+\dist(z,\X^c))\leq \frac{\eta}{2}(t+\dist(z,\X^c)),
\end{align*}
which after rearranging and using \eqref{e:johncurve} yields $\eta> 1$, a contradiction. Thus we must have~\eqref{e:ratiodisteta}.

For purely notational purposes, we will now consider the reverse path $\tilde{\gamma}$ defined by $\tilde{\gamma}(t)=\gamma(\ell(\gamma)-t)$. Letting $L=1000$, we set $t_0=0$, $y_0=\tilde{\gamma}(t_0)=x_0$, and recursively define 
$$
t_i:=t_{i-1}+\frac{1}{L} \delta(y_{i-1}), \qquad y_i:=\tilde{\gamma}(t_i)
$$
for $i\in\N$. The inequality \eqref{e:ratiodisteta} implies that $\delta(\gamma(t))$ is bounded below over $t\in[0,\ell(\gamma)]$, hence there exists a smallest $N\in\N\cup\{0\}$ such that $t_{N+1}>\ell(\gamma)$. For each $i\in\{0,\ldots,N\}$ we consider $Q_i\in\mathcal{F}$ such that
\begin{equation}\label{e:martinet}
    B(y_i,\delta(y_i)/100)\subset Q_i,
\end{equation}
which exists according to \eqref{e:FkGk}. We also claim it is possible to take $Q_N=Q$. To see this, we observe that $\dist(x,y_N)\leq \frac{1}{L}\delta(y_N)$ which implies $\delta(y_N)\leq \delta(x)+\dist(x,y_N)\leq \delta(x)+\frac{1}{L}\delta(y_N)$, hence $\delta(y_N)\leq \frac{L}{L-1}\delta(x)$, thus $B(y_N,\delta(y_N)/100)\subset Q$. Note that $Q_0,\ldots,Q_N$ are not necessarily distinct, and that $y_i$ is not necessarily the center of $Q_i$.

Let us prove \eqref{e:BJohn}. For this we first prove that if $j\in\{0,\ldots,N\}$, then 
\begin{equation}\label{e:distballbord}
\dist(Q_j,\mathcal{X}^c)\leq c_Mr_j
\end{equation}
where $c_M:=20\frac{{\rm diam}(M)}{R}-1$, $x_j$ is the center of $Q_j$ (in particular, $x_N=x$ is the center of $Q$), and $r_j=\delta(x_j)/20$ denotes the radius of $Q_j$. 
Indeed, $\dist(Q_j,\X^c)=\dist(x_j,\X^c)-r_j$, hence \eqref{e:distballbord} follows since
$$
\frac{\dist(Q_j,\X^c)}{r_j}=\frac{\dist(x_j,\X^c)}{r_j}-1=20\frac{\dist(x_j,\X^c)}{\min(\dist(x_j,\X^c),R)}-1\leq c_M.
$$
We also deduce from \eqref{e:distballbord}
\begin{equation}\label{e:distyjbord}
\dist(x_j,\X^c)\leq (c_M+1)r_j.
\end{equation}
Now fix $j\in\{0,\ldots,N\}$, then the distance from $x_j$ to an arbitrary point of $Q=Q_N$ is at most $\dist(x_j,x_N)+r_N$. Therefore to prove \eqref{e:BJohn} it suffices to show that there exists $\B$ depending only on $M$ such that
$$
\dist(x_j,x_N)+r_N\leq \B r_j.
$$
Using \eqref{e:ratiodisteta}, \eqref{e:distballbord} and recalling that $y_j$ belongs to both $Q_j$ and the image of $\gamma$, we have
\begin{equation}\label{e:rQleqdist}
r_N=\frac{1}{20} \delta(x_N)\leq \frac{1}{20} \dist(x_N,\mathcal{X}^c)\leq \frac{1}{10\eta} \dist(y_j,\mathcal{X}^c)\leq \frac{1}{10\eta}(\dist(x_j,\X^c)+r_j),
\end{equation}
while using \eqref{e:johncurve},
\begin{equation}\label{e:maxim}
\dist(x_j,x_N)\leq \dist(y_j,x_N)+r_j\leq \frac{1}{\eta}\dist(y_j,\X^c)+r_j\leq \frac{1}{\eta}(\dist(x_j,\X^c)+r_j)+r_j.
\end{equation}
Thus, \eqref{e:rQleqdist}, \eqref{e:maxim} and \eqref{e:distyjbord} yield
$$
\dist(x_j,x_N)+r_N\leq \frac{11\dist(x_j,\mathcal{X}^c)+21r_j}{10\eta} \leq  \frac{11c_M+32}{10\eta} r_j
$$
which proves \eqref{e:BJohn}.

Finally, let us show \eqref{e:conseccubeoverlap}. Fix $j\in\{0,\ldots,N-1\}$. We first prove that 
\begin{equation}\label{e:inclu}
B(y_j,\delta(y_j)/200)\subset Q_j\cap Q_{j+1}.
\end{equation}
By~\eqref{e:martinet}, we only need to show that $B(y_j,\delta(y_j)/200)\subset Q_{j+1}$. Let $z\in B(y_j,\delta(y_j)/200)$, then using that $\delta(y_{j+1})\geq \frac{L-1}{L} \delta(y_j)$ and \eqref{e:martinet}, we have
\begin{align*}
\dist(x_{j+1},z)&\leq \dist(x_{j+1},y_{j+1})+\dist(y_{j+1},y_j)+\dist(y_j,z)\\
&\leq \frac{1}{20}\delta(x_{j+1})-\frac{1}{100}\delta(y_{j+1})+\frac{1}{L}\delta(y_j)+\frac{1}{200}\delta(y_j)\\
&\leq \frac{1}{20}\delta(x_{j+1}),
\end{align*}
where recall that $L=1000$.
Since $Q_{j+1}=B(x_{j+1},\delta(x_{j+1})/20)$, this implies \eqref{e:inclu}. Using similarly several times the triangle inequality we obtain
\begin{equation}\label{e:qqq}
Q_j\cup Q_{j+1}\subset B(y_j,\delta(y_j)/2).
\end{equation}
Finally combining \eqref{e:inclu} and \eqref{e:qqq}, we deduce
$$
\rho(Q_j\cap Q_{j+1})\geq \rho(B(y_j,\delta(y_j)/200))\geq C\rho(B(y_j,\delta(y_j)/2))\geq C\max(\rho(Q_j),\rho(Q_{j+1}))
$$
where $C$ only depends on the doubling constant of the metric space.
\end{proof}

\begin{lemma}\cite[Lemma 3.3]{letrouitmerigot} \label{l:gluevarjohnmanifolds} 
Assume that $\rho$ is a probability measure on $\X$ which is the restriction to an open set $\X$ of a doubling measure on a metric space, and satisfies the Boman chain condition for some covering $\mathcal{F}$ and some $\A$, $\B$, $\C>1$. 
Then there exists $\kappa>0$ such that for any continuous function $f$ on $\X$, there holds 
$$
{\rm Var}_\rho(f)\leq \kappa\sum_{Q\in\mathcal{F}}\rho(Q)\Var_{\tilde{\rho}_Q}(f)
$$
where $\tilde{\rho}_Q=\frac{1}{\rho(Q)}\rho_{|Q}$.
\end{lemma}
This lemma is proved in \cite[Section 3.2 and Appendix A]{letrouitmerigot} for $M=\R^d$ but the proof carries over without any change to any smooth, $d$-dimensional Riemannian manifold $M$. 

\subsubsection{Strongly convex functions}
Another preliminary step in proving Proposition \ref{p:stabpote} is to prove that the function $z\mapsto \dist(z,x)^2$ is $\theta$-strongly convex in some ball $B(x,R)$, where both $\theta$ and $R$ are uniform over $x\in\X$. 
\begin{proposition} \label{p:strongconvfunc}
There exist $\theta,R>0$ such that for any $x\in\X$, the function $z\mapsto \dist(z,x)^2$ is $\theta$-strongly convex in $B(x,R)$. 
\end{proposition}
\begin{proof}
Fix $x\in\X$ and let $V_x:z\mapsto \frac12 d(z,x)^2$. Also fix $R<{\rm inj}(M)$, denote by $B_x\subset T_xM$ the ball of radius $R$ centered at the origin of $T_xM$, and by $B(x,R)\subset M$ the ball of radius $R$ centered at $x$ in $M$. Let $h:B_x\rightarrow \R$ be defined by $h(v):= \frac12 |v|^2_{g_x}$. The exponential map at $x$ is a $C^\infty$ diffeomorphism from $B_x$ to $B(x,R)$, whose inverse is denoted by $\ell:B(x,R)\rightarrow B_x$; since the exponential map is an isometry in the radial direction, we have $V_x=h\circ \ell$. Since $h$ is also smooth, we find that $V_x=h\circ \ell$ is smooth in $B(x,R)$. \\ 
Let $z\in B(x,R)$, $v\in T_zM$, and let $(\gamma_t)$ denote the geodesic with $\gamma_0=z$ and $\left.\frac{\dd \gamma_t}{\dd t}\right\vert_{t=0}=v$. Then
\begin{equation}\label{e:cettehessienne}
\begin{aligned}
&h(\ell(\gamma_t))-h(\ell(\gamma_0))\\ &\qquad=\frac12 \Bigl|\ell(z)+tD\ell(z)v + \frac{t^2}{2}D^2\ell(z)(v,v)+O(t^3)\Bigr|_{g_x}^2-\frac12 |\ell(z)|^2_{g_x}\\
&\qquad =tg_x(\ell(z),D\ell(z)v) + \frac12 t^2\left(|D\ell(z) v|_{g_x}^2+\langle \ell(z),D^2\ell(z)(v,v)\rangle_{g_x}\right)+O(t^3)
\end{aligned}
\end{equation}
as $t\rightarrow 0$.
We have $\ell(x)=0$, and the Hessian of $V_x=h\circ\ell$ at $z$ is $g_x(D\ell(z)\cdot, D\ell(z)\cdot)+O(|\ell(z)|)$ as a consequence of smoothness of $\ell$ and \eqref{e:cettehessienne}. We deduce that ${\rm Hess}(V_x)\geq \theta_x {\rm Id}$ in $B_x$, for some $\theta_x>0$. Since $g$ is a smooth metric, we see that $\theta_x$ may even be chosen independent of $x$ in sufficiently small balls. Thus using compactness of $\overline{\X}$, we obtain the result.
\end{proof}

\subsubsection{End of the proof of Proposition \ref{p:stabpote}} \label{sec: potential stability proof}
We assume $c(x,y)=\frac12 \dist(x,y)^2$. Our proof of Proposition \ref{p:stabpote} consists in obtaining bounds on the variance in balls $Q_i$ (whose radius is sufficiently small, in order to guarantee existence of a strongly convex function) with respect to some measure $\rho_i^V$ defined in \eqref{e:rhoiV}, and then gluing these local variance bounds together to obtain a bound on the variance with respect to $\rho$ via Lemma \ref{l:gluevarjohnmanifolds}.

Let $\mathcal{F}$ be an open cover of $\X$ given by Proposition~\ref{prop: boman chain construction}, with $R$ chosen as in Proposition \ref{p:strongconvfunc}. Since $\mathcal{F}$ is a countable set of balls, we may consider an enumeration $\mathcal{F}=\{Q_i\}_{i=1}^\infty$, where $Q_i=B(x_i,r_i)$, $x_i\in\X$, and $r_i\leq R$. We set $\rho_i:=\frac{1}{\rho(Q_i)}\rho_{|Q_i}$. Since $r_i\leq R$, the function $z\mapsto \dist(z,x_i)^2$ is $\theta$-strongly convex on $Q_i$ according to Proposition \ref{p:strongconvfunc}. We choose $K>0$ large enough so that $V_i=K\dist(\cdot,x_i)^2$ satisfies
\begin{equation}\label{e:hess+ric}
{\rm Hess\ } V_i+{\rm Ric}\geq \lambda
\end{equation} 
on $Q_i$. Notice that since ${\rm Ric}$ is uniformly bounded below, $K$ may be chosen independent of $i$.
We then consider
\begin{equation}\label{e:rhoiV}
\rho^V_i:=\frac{\exponentiate{-V_i}\chi_{Q_i}}{\int_{Q_i}\exponentiate{-V_i}{\rm dvol}}\dd {\rm vol}
\end{equation}
which is a log-concave probability measure on $Q_i$.
Since
$$
\frac{m_\rho}{M_\rho}\dd {\rm vol}\leq \rho_i {\rm vol}(Q_i)\leq \frac{M_\rho}{m_\rho}\dd {\rm vol} \qquad \text{and} \qquad \exponentiate{-Kr_i^2}\dd {\rm vol}\leq \rho^V_i {\rm vol}(Q_i)\leq \exponentiate{Kr_i^2}\dd {\rm vol}
$$
on $Q_i$, we have
\begin{equation}\label{e:comparrhoi}
\E^{-1}\rho_i\leq \rho^V_i \leq \E \rho_i,\qquad \E:=\exponentiate{K R^2}\frac{M_\rho}{m_\rho}.
\end{equation}

In the sequel, when applying the results of Section \ref{s:convkantoreg} we will fix $\sigma$ to be a Borel probability measure whose support coincides with $\Y$ (the existence of such $\sigma$ is easily seen by picking a dense and countable set in $\Y$, and constructing $\sigma$ as a sum of weighted Dirac deltas on this set). Using \eqref{eq:diffpsieps} we compute
\begin{equation} \label{e:psi1psi0eps}
\begin{aligned}
\Var_{\rho^V_i}(\psi_1^{c,\varepsilon}-\psi_0^{c,\varepsilon})=\Var_{\rho^V_i}\Bigl(\int_0^1 \frac{\dd}{\dd t}\psi_t^{c,\varepsilon} \dd t\Bigr)&\leq \int_0^1\Var_{\rho^V_i}\Bigl(\frac{\dd}{\dd t}\psi_t^{c,\varepsilon}\Bigr)\dd t\\
&=\int_0^1 \Var_{\rho^V_i}(x\mapsto \sca{\mu_\eps^x[\psi_t]}{v})\dd t.
\end{aligned}
\end{equation}
The next step is to use Proposition \ref{p:D2K}; to do so we must verify Assumption \ref{a:assumptions}. 
The fact that the cost $c(x,y)=\frac12 \dist(x,y)^2$ is semi-concave follows from \cite[Lemma 3.3]{ohta} (using that $\X$ is bounded) since any Riemannian manifold is an Alexandrov space with same curvature bounds,
 while \eqref{e:hess+ric} implies the bound on the $\infty$-Bakry--Emery tensor, and we take $U=Q_i\subset M$ for each $i$ which is geodesically convex, thus Assumption \ref{a:assumptions} is verified with uniform constants.

Hence using \eqref{e:psi1psi0eps}, Proposition \ref{p:D2K} (and its constant $C_0$) and the identity $\sca{D^2\mathcal{K}^\varepsilon_{\rho^V_i}[\psi_t]v}{v}=\frac{\dd}{\dd t}\sca{\nabla \mathcal{K}_{\rho^V_i}^\varepsilon[\psi_t]}{v}$ where $v=\psi_1-\psi_0$, we get
\begin{equation}\label{e:Varrhoi}
\begin{aligned}
\Var_{\rho^V_i}(\psi_1^{c,\varepsilon}-\psi_0^{c,\varepsilon})\leq -C_0^2\int_0^1\langle D^2\mathcal{K}^\varepsilon_{\rho^V_i}[\psi_t]v,v\rangle \dd t&=C_0^2\langle \nabla \mathcal{K}_{\rho^V_i}^\eps[\psi_0]-\nabla \mathcal{K}_{\rho^V_i}^\eps[\psi_1],\psi_1-\psi_0\rangle\\
&=C_0^2 \sca{\mu_{\eps, i}^V[\psi_1]-\mu_{\eps, i}^V[\psi_0]}{\psi_1-\psi_0}
\end{aligned}
\end{equation}
where in the last line we used Lemma \ref{l:gradient} and $\mu_{\eps, i}^V[\psi]$ is defined by
\begin{equation}\label{e:muepsi}
 \forall v \in\Class^0(M), \quad \sca{\mu_{\eps, i}^V[\psi]}{v} := \int_{Q_i} \sca{\mu_\eps^x[\psi]}{v} \dd \rho^V_i(x).
\end{equation}
We sum \eqref{e:Varrhoi} with a weight $\delta_i^V:=\int_{Q_i}\exponentiate{-V_i}\dd {\rm vol}$, we get
\begin{equation}\label{e:deltaikrhoi}
\begin{aligned}
\sum_{Q_i\in\mathcal{F}} \delta_i^V\Var_{\rho^V_i}(\psi_1^{c,\varepsilon}-\psi_0^{c,\varepsilon})&\leq C_0^2\sum_{Q_i\in\mathcal{F}} \delta_i^V  \sca{\mu_{\eps, i}^V[\psi_1]-\mu_{\eps, i}^V[\psi_0]}{\psi_1-\psi_0}\\
&=C_0^2\sum_{Q_i\in\mathcal{F}} \int_{Q_i}\sca{\mu_\varepsilon^x[\psi_1]-\mu_\varepsilon^x[\psi_0]}{\psi_1-\psi_0} \exponentiate{-V_i(x)}\dd {\rm vol}(x)\\
&\leq \exponentiate{KR^2}m_\rho^{-1}C_0^2\sum_{Q_i\in \mathcal{F}} \int_{Q_i}\sca{\mu_\varepsilon^x[\psi_1]-\mu_\varepsilon^x[\psi_0]}{\psi_1-\psi_0} \dd \rho(x)\\
&\leq A\E C_0^2\sca{\mu_{\eps}[\psi_1]-\mu_{\eps}[\psi_0]}{\psi_1-\psi_0},
%&=C_0\langle \psi_1-\psi_0,\bar{\mu}_{\eps}[\psi_1]-\bar{\mu}_{\eps}[\psi_0]\rangle
\end{aligned}
\end{equation}
where we recall $\sca{\mu_\eps[\psi]}{v} = \int_X \sca{\mu_\eps^x[\psi]}{v} \dd \rho(x)$.
To pass from the second to third line we have used that for any $x\in\X$,
$$
\langle \mu_\varepsilon^x[\psi_1]-\mu_\varepsilon^x[\psi_0],\psi_1-\psi_0\rangle\geq 0,
$$
which follows from the concavity of $\mathcal{K}^\eps_{\delta_x}$ (obtained by applying Lemma \ref{l:hesskanto} with the choice $\rho=\delta_x$) by a calculation similar to \eqref{e:Varrhoi}. Using \eqref{e:comparrhoi}, it follows from \eqref{e:deltaikrhoi} that 
$$
\sum_{Q_i\in \mathcal{F}} \rho(Q_i)\Var_{\rho_i}(\psi_1^{c,\varepsilon}-\psi_0^{c,\varepsilon})\leq A\E^3 C_0^2\sca{\mu_{\eps}[\psi_1]-\mu_{\eps}[\psi_0]}{\psi_1-\psi_0}.
$$

Thanks to Lemma \ref{l:gluevarjohnmanifolds}, we deduce that
$$
\Var_\rho(\psi_1^{c,\varepsilon}-\psi_0^{c,\varepsilon})\leq \kappa A\E^3C_0^2\langle \mu_{\eps}[\psi_1]-\mu_{\eps}[\psi_0], \psi_1-\psi_0\rangle,
$$
we will now take the limit $\varepsilon\rightarrow 0$ in this inequality. For $i=1$, $2$, we have $\psi_i^{c,\varepsilon}\rightarrow \psi_i^c$ pointwise as $\varepsilon\rightarrow 0$ (see Remark \ref{r:convergence}). Since $\psi_i^{c,\varepsilon}$ are Lipschitz uniformly in $\varepsilon$, we obtain that the convergence is uniform over any compact set, in particular over $x$ in the support of $\rho$, hence 
$$
\Var_\rho(\psi_1^{c,\varepsilon}-\psi_0^{c,\varepsilon})\to \Var_\rho(\psi_1^{c}-\psi_0^{c})
$$
as $\varepsilon\to 0$. Let us show that $\mu_\eps[\psi_i]\rightarrow  (y_{\psi_i})_{\#}\rho$ as $\varepsilon\rightarrow 0$. For this, we consider $\gamma^\varepsilon_i\in \mathcal{P}(\X\times \Y)$ defined by 
$$
\frac{\dd\gamma^\varepsilon_i}{\dd(\rho\otimes\sigma)}(x,y)= \hat{\mu}_\eps^x[\psi_i](y).
$$
This means that for any continuous function $\eta$ on $\Class^0(\overline{\X}\times \Y),$ we have
\begin{align*}
\int_{\X\times \Y} \eta(x,y) \dd \gamma^\varepsilon_i(x,y) 
&= \int_{\X\times \Y} \eta(x,y) \hat{\mu}_\eps^x[\psi_i] \dd[\rho\otimes \sigma](x,y)\\
&= \int_{\X} \int_\Y \eta(x,y) \frac{\exponentiate{-\frac{c(x,y) - \psi_i(y)}{\eps}}}
{\int_\Y \exponentiate{-\frac{c(x,z) - \psi_i(z)}{\eps}} \dd\sigma(z)} \dd\sigma(y) \dd \rho(x)  
\end{align*}
By the Laplace principle, and using that $\sigma$ has full support, we obtain that for every $x\in \X$ such that the minimum of $y\mapsto c(x,y) - \psi_i(y)$ is attained at a unique point $y_{\psi_i}(x)$, we have 
$$ \lim_{\eps\to 0} \int_\Y \eta(x,y) \frac{\exponentiate{-\frac{c(x,y) - \psi_i(y)}{\eps}}}
{\int_\Y \exponentiate{-\frac{c(x,z) - \psi_i(z)}{\eps}} \dd\sigma(z)} \dd\sigma(y) = \eta(x,y_{\psi_i}(x)).$$
Moreover (see e.g. \cite[Lemmas 7 and 4]{mccann}), this minimum is attained at a unique point for $\rho$-almost every $x$. Using the dominated convergence theorem we obtain that 
$$ \lim_{\eps\to 0} \int_{\X\times\Y} \eta(x,y) \dd \gamma^\varepsilon_i(x,y) = \int_\X \eta(x,y_{\psi_i}(x)) d\rho(x),  $$
thus proving that $\gamma^\varepsilon_i$ converges narrowly to $\gamma_i = (\Id,y_{\psi_i})_{\#}\rho.$
This implies that the first marginal  $\mu_{\eps}[\psi_i]$ of $\gamma_i^\eps$ converges narrowly to $\mu_i = y_{\psi_i\#}\rho$ for $i\in\{0,1\}$, so that 
$$ \lim_{\eps\to 0} \sca{\mu_{\eps}[\psi_1] - \mu_{\eps}[\psi_0]}{\psi_1 - \psi_0} = \sca{\mu_0 - \mu_1}{\psi_1-\psi_0}.$$

\section{Stability of optimal transport maps} \label{s:stabmapmanifo}
To complete the proof of Theorem \ref{t:stabpote2}, we show in this section that \eqref{e:mapsstab} holds. The proof consists of two steps. In the first step, we notice that
\begin{equation}\label{e:dilatedist}
\forall x\in \X, \qquad \dist(\exp_x(-\nabla \phi_\mu(x)),\exp_x(-\nabla \phi_\nu(x)))\leq C |\nabla \phi_\mu(x)-\nabla \phi_\nu(x)|_{g_x}
\end{equation}
for some $C>0$ which only depends on $\X,\Y$.
Indeed, since the norms of $\nabla \phi_\mu$ and $\nabla \phi_\nu$ are pointwise bounded by $L={\rm diam}(\Y)$ (according to \cite[Lemma 2]{mccann}), \eqref{e:dilatedist} is a direct consequence of the smoothness of the exponential map and compactness of the set of tangent vectors of norm bounded by $L$ above $\X$, since $\X$ is bounded.

In the second step of the proof of \eqref{e:mapsstab}, we prove the inequality
\begin{equation}\label{e:gn}
\int_{\X}|\nabla \phi_\mu(x)-\nabla\phi_\nu(x)|_{g_x}^2 d\rho(x)\leq C \left(\int_{\X} |\phi_\mu(x)-\phi_\nu(x)|^2  d\rho(x)\right)^{1/3}
\end{equation}
where $C$ is a constant which depends only on $\X$ and $\rho$. We consider the geodesic flow at time $s$ on the sphere bundle $SM$, denoted by $\Phi_s: SM\rightarrow SM$. We will also decompose $\Phi_s$ as $\Phi_s(x,v)=(b_s(x,v),t_s(x,v))$ with $b_s(x,v)\in M$ (the base point) and $t_s(x,v)\in S_{b_s(x,v)}M$ (the tangent vector). Fix $T>0$ with $T<{\rm inj}(M)/2$; the latter condition will be used in the proof of Lemma \ref{l:semicon} below. For a fixed $(x,v)\in SM$, we write the parameters where $\X$ intersects the image of a geodesic with initial conditions $(x, v)$ on the interval $[0, T]$, as a union of open intervals indexed by a set $I_T^\X(x,v)$:
$$
\{s\in [0,T] \mid b_s(x,v)\in \X\}= \bigcup_{i\in I_T^\X(x,v)} (\alpha_i(x,v),\beta_i(x,v)).
$$
Let
$$
f(x,v):=
\begin{cases} 
g(\nabla \phi_\mu(x)-\nabla\phi_\nu(x), v)^2, & x\in\X,\\
0,& \text{otherwise}.
\end{cases}
$$
We notice that $\int_{S_xM}f(x,v)dv=c_d|\nabla \phi_\mu(x)-\nabla\phi_\nu(x)|_{g_x}^2$ for some constant $c_d$ which only depends on the dimension $d$ (and in particular, is independent of $x$). 
Using the invariance of the Liouville measure $m$ on $SM$ under the geodesic flow $(\Phi_s)$, and the fact that the density of $\rho$ is bounded above by a positive constant, we obtain
\begin{equation}\label{e:nablaphi1}
\begin{aligned}
&\int_{\X} |\nabla\phi_\mu(x)-\nabla\phi_\nu(x)|_{g_x}^2\dd\rho(x)\\
&\qquad\leq C\int_{SM}f(x,v)\dd m(x,v)\\
&\qquad=CT^{-1}\int_{SM}\int_0^T f(\Phi_s(x,v))\dd s\dd m(x,v)\\
&\qquad=CT^{-1}\int_{SM} \sum_{i\in I_T^\X(x,v)} \int_{\alpha_i(x,v)}^{\beta_i(x,v)}g( \nabla\phi_\mu(b_s(x,v))-\nabla\phi_\nu(b_s(x,v)), t_s(x,v))^2\dd s\dd m(x,v).
\end{aligned}
\end{equation}
For $(x,v)\in SM$, we let $u_\mu^{(x,v)}(s)=\phi_\mu(b_s(x,v))$ for $s\in[0,T]$. Then 
$$
\frac{\dd}{\dd s}u_\mu^{(x,v)}(s)=g( \nabla\phi_\mu(b_s(x,v)), t_s(x,v)),
$$
and an analogous formula holds for $u_\nu^{(x,v)}(s)=\phi_\nu(b_s(x,v))$. Denoting by $\X_T$ the set of $x\in M$ at distance at most $T$ from the set $\X$, we notice that in the last line of \eqref{e:nablaphi1}, only the elements $(x,v)\in SM$ for which $x \in\X_T$ have a non-vanishing contribution.
\begin{lemma}\label{l:semicon}
There exists $\zeta$ independent of $(x,v)\in S\X_T$ such that for any $(x,v)\in S\X_T$, the function $s\mapsto u^{(x,v)}_\mu(s)-\zeta|s|^2$ is concave on $[0,T]$ and the modulus of its derivative (which exists a.e. on $[0, T]$ as a result) is bounded above by ${\rm diam}(\mathcal{Y})+2\zeta T$ on $[0,T]$.
\end{lemma} 
\begin{proof}
Note that 
\begin{align*}
    u_\mu^{(x, v)}(s)=\inf_{y\in \mathcal{Y}}(\frac{1}{2}\dist(b_s(x, v), y)^2-\psi_\mu(y)).
\end{align*}
Then since the sectional curvature of $M$ is bounded from below on the bounded set $\X_T$, by \cite[Lemma 3.3]{ohta} there exists a $\tilde\zeta\in \R$ independent of $(x,v)\in S\X_T$ and $y\in\mathcal{Y}$ such that
\begin{align*}
    \dist(b_s(x, v), y)^2
    &\geq \left(1-\frac{s}{T}\right)\dist(x, y)^2+\left(\frac{s}{T}\right)\dist(b_T(x, v), y)^2-\tilde\zeta\left(1-\frac{s}{T}\right)\left(\frac{s}{T}\right)\dist(x, b_T(x, v))^2\\
    &= \left(1-\frac{s}{T}\right)\dist(x, y)^2+\left(\frac{s}{T}\right)\dist(b_T(x, v), y)^2-\tilde\zeta\left(1-\frac{s}{T}\right)\left(\frac{s}{T}\right)T^2,
\end{align*}
where we have used that $s\mapsto b_s(x, v)$ is a unit speed geodesic in $M$ starting from $x$ in the last line. Thus if we let $\zeta:=\tilde \zeta/2$, we see the function $s\mapsto \frac{1}{2}\dist(b_s(x, v), y)^2-\psi_\mu(y)-\zeta\lvert s\rvert^2$ is concave on $[0, T]$ for any $y\in \mathcal{Y}$. Thus as an infimum of concave functions, we find $s\mapsto u_\mu^{(x, v)}(s)-\zeta\lvert s\rvert^2$ is also concave on $[0, T]$.

Using that $\phi_\mu$ is ${\rm diam}(\mathcal{Y})$-Lipschitz (see \cite[Lemma 1]{mccann}), we obtain the bound on the modulus of the derivative.
\end{proof}
We also recall the following result (for a proof, see \cite[Lemma 5.1]{delalandemerigot}).
\begin{lemma}\label{l:1d}
Let $I\subset\R$ be a compact segment and let $u$, $v:I\rightarrow\R$ be two
convex functions whose derivatives (defined a.e. on $I$) are uniformly bounded on $I$. Then
$$
\|u'-v'\|_{L^2(I)}^2\leq 8(\|u'\|_{L^\infty(I)}+\|v'\|_{L^\infty(I)})^{4/3}\|u-v\|_{L^2(I)}^{2/3}.
$$
\end{lemma}
We apply Lemma \ref{l:1d} to the functions $\zeta|\cdot|^2-u_\mu^{(x,v)}$ and $\zeta|\cdot|^2-u_\nu^{(x,v)}$ on each segment $[\alpha_i(x,v),\beta_i(x,v)]$, using that the derivatives in $s$ of these functions is bounded according to Lemma \ref{l:semicon}. Combining with \eqref{e:nablaphi1}, we get 
\begin{align}
&\int_{\X} |\nabla\phi_\mu(x)-\nabla\phi_\nu(x)|_{g_x}^2\dd\rho(x)\nonumber\\
&\qquad \leq C\int_{SM}\sum_{i\in I_T^\X(x,v)}\Bigl(\int_{\alpha_i(x,v)}^{\beta_i(x,v)} |\phi_\mu(b_s(x,v))-\phi_\nu(b_s(x,v))|^2\dd s\Bigr)^{1/3}\dd m(x,v)\nonumber\\
&\qquad \leq C\int_{SM}(\# I_T^\X(x,v))^{2/3}\Bigl(\int_0^T|\phi_\mu(b_s(x,v))-\phi_\nu(b_s(x,v))|^2 \dd s \Bigr)^{1/3} \dd m(x,v)\nonumber\\
&\qquad\leq C\Bigl(\int_{SM}\# I_T^\X(x,v)\dd m(x,v)\Bigr)^{2/3}\Bigl(\int_{SM} \int_0^T|\phi_\mu(b_s(x,v))-\phi_\nu(b_s(x,v))|^2 \dd s \dd m(x,v)\Bigr)^{1/3}\nonumber\\
&\qquad \leq C\Bigl(\int_{SM}\# I_T^\X(x,v)\dd m(x,v)\Bigr)^{2/3}\Bigl(\int_{M} |\phi_\mu(x)-\phi_\nu(x)|^2 \dd\rho(x)\Bigr)^{1/3}\label{e:allezlOM}
\end{align}
using Hölder's inequality for counting measure to obtain the third line, H\"older's inequality for $m$ to obtain the fourth line, and in the last line the invariance of $m$ under the geodesic flow and the fact that the density of $\rho$ is bounded below.

To control the first integral in \eqref{e:allezlOM} we will show the following proposition.
\begin{proposition}\label{p:zelditch}
Assume that $\mathscr{H}^{d-1}(\partial\X)<+\infty$, and let $T<{\rm inj}(M)$. Then 
\begin{equation}\label{e:zelditch}
\int_{SM}\# I_T^\X (x,v)\dd m(x,v)<+\infty.
\end{equation}
\end{proposition}

Proposition \ref{p:zelditch} together with \eqref{e:allezlOM} concludes the proof of \eqref{e:gn}. Then combining with \eqref{e:varinrho0} and \eqref{e:dilatedist}, we obtain \eqref{e:mapsstab}, finishing the proof of our main theorem. 

The above proposition is closely related to a result by Zelditch (see \cite[Proposition 9]{zelditch}), which however only applies in the case where $\partial\X$ is smooth (we do note that Zelditch mentions, without proof, that the result extends to the case where $\partial\X$ has a singular set of Hausdorff dimension at most $d-2$).

\begin{proof}[Proof of Proposition \ref{p:zelditch}]
We follow Caratheodory's construction (see \cite[Chapter 2.10]{federer}). Given a Borel set $S\subset M$, we set
$$
\zeta(S):=m\Bigl(\bigcup_{t\in[0,T]}\Phi_t(\pi^{-1}(S))\Bigr)
$$
where $\pi:SM\rightarrow M$ denotes the canonical projection from the unit tangent bundle $SM$ onto $M$, and $m$ again denotes the Liouville measure on $SM$. For $E\subset M$, we also define
$$
\IG_\delta(E):=\inf\sum_{S\in G} \zeta(S) 
$$
where the infimum above is taken over all countable families $G$ of Borel sets in $M$ with diameter at most $\delta$, satisfying $E\subset \bigcup_G S$. The fact that $\IG_\delta\geq \IG_\sigma$ for $0<\delta\leq \sigma\leq \infty$ implies the existence of the ``integral geometric" measure
$$
\IG(E):=\lim_{\delta\rightarrow 0^+}\IG_\delta(E)=\sup_{\delta>0} \IG_\delta(E) \qquad \text{for } E\subset M.
$$

We claim that for any bounded subset $K\subset M$, there exists $C>0$ such that for any $E\subset K$, 
\begin{equation}\label{e:IGHaus}
\IG(E)\leq C\mathscr{H}^{d-1}(E).
\end{equation}
To prove \eqref{e:IGHaus}, by the definition of $(d-1)$-dimensional Hausdorff measure, it is sufficient to prove that there exists $C>0$ such that for any Borel set $S\subset K$ it holds that $\zeta(S)\leq C{\rm diam}(S)^{d-1}$; and in fact without loss of generality it is sufficient to prove this for all $S$ with diameter less than $T$. Let $S\subset K$ be such a Borel set.
Let $\delta:={\rm diam}(S)$ and let $N\in\N$ be such that $T/N \leq \delta\leq T/(N+1)$. We set
$$
F_k:=\bigcup_{t\in[k\delta,(k+1)\delta]}\Phi_t(\pi^{-1}(S))
$$
for $k=0,\ldots,N-1$. Obviously, $F_0$ is contained in a set of diameter at most $4\delta$, hence $m(F_0)\leq C\delta^d$ for some $C$ depending only on a lower bound on the curvature of $M$ in the bounded set $K$. Moreover, since $F_k=\Phi_{k\delta}(F_0)$ for any $k$, it holds that $m(F_k)=m(F_0)$ by invariance of the Liouville measure under the geodesic flow. Hence $\zeta(S)\leq CN\delta^d=CT\delta^{d-1}\leq C'{\rm diam}(S)^{d-1}$, which concludes the proof of \eqref{e:IGHaus}.

Now, let us show that 
\begin{equation}\label{e:formulaforIG}
    \IG(E)=\int_{SM} \#\{t\in[0,T] \mid b_t(x,v)\in E\}\  \dd m(x,v)
\end{equation}
where recall that $b_t(x,v)=\pi(\Phi_t(x,v))$.
We choose Borel partitions $H_1,H_2,H_3,\ldots$ of $E$ such that each member of $H_j$ is the union of some subfamily of $H_{j+1}$ (i.e. $H_{j+1}$ is a refinement of $H_j$), and
\begin{equation}\label{e:supdiam}
\sup_{S\in H_j} {\rm diam }(S)\underset{j\rightarrow +\infty}{\longrightarrow} 0.
\end{equation}
Due to \cite[Theorem 2.10.8]{federer},
\begin{equation}\label{e:smallargument}
\sum_{S\in H_j} \zeta(S)\underset{j\rightarrow +\infty}{\longrightarrow} \IG(E).
\end{equation}
Setting
$$
A_S^T:=\{(x,v)\in SM \mid \exists t\in[0,T],\ b_t(x,v)\in S\},
$$
then for any $j$, using Tonelli's theorem,
\begin{align}
\sum_{S\in H_j}\zeta(S)&=\sum_{S\in H_j}m\left(\bigcup_{t\in [0, T]}\Phi_t(\pi^{-1}(S))\right)\nonumber\\
&=\sum_{S\in H_j}\int_{SM} \chi_{A_S^T}(x,v) \dd m(x,v)\nonumber\\
&=\int_{SM} \#\{S\in H_j \mid \exists t\in[0,T],\ b_t(x,v)\in S\}\dd m(x,v)\label{e:SinHj}
\end{align}
Let us show that
\begin{equation}\label{e:SinHj2}
\Xi_j:=\ \#\{S\in H_j \mid \exists t\in[0,T],\ b_t(x,v)\in S\}\underset{j\rightarrow +\infty}{\longrightarrow}\#\{t\in[0,T] \mid b_t(x,v)\in E\} =:\Xi_\infty
\end{equation}
for any $(x,v)\in SM$. First, suppose $\Xi_\infty<\infty$, thus there exist $0\leq t_1<t_2<\ldots<t_{\Xi_\infty}\leq T$ such that $b_{t_i}(x,v)\in E$ for all $i=1,\ldots,\Xi_\infty$; since $T<{\rm inj}(M)$ we see all $b_{t_i}(x,v)$ are distinct. Then for all $j$ large enough, by \eqref{e:supdiam} we have 
\begin{align*}
    \sup_{S\in H_j}{\rm diam}(S)<\min_{1\leq i_1, i_2\leq \Xi_\infty} \dist(b_{t_{i_1}}(x,v), b_{t_{i_2}}(x,v)),
\end{align*}
and since $H_j$ consists of mutually disjoint sets this implies exactly $\Xi_\infty$ of the sets from $H_j$ contain $b_t(x, v)$ for some $t\in [0, T]$. Otherwise $\Xi_\infty=+\infty$, then for any $N$ we can find $0\leq t_1<t_2<\ldots<t_N\leq T$ such that $b_{t_i}(x,v)\in E$ for all $i=1,\ldots,N$. Again by \eqref{e:supdiam} and the fact that all $b_{t_i}(x,v)$ are distinct we also obtain the convergence \eqref{e:SinHj2} in this case. Since each $H_{j+1}$ is a refinement of $H_j$, we see that $\Xi_j$ is increasing in $j$ and clearly each $\Xi_j\geq 0$, thus we can combine \eqref{e:smallargument}, \eqref{e:SinHj}, \eqref{e:SinHj2} with the monotone convergence theorem to conclude that \eqref{e:formulaforIG} holds.

Finally, 
$$
\# I_T^{\X}(x,v)\leq \#\{t\in[0,T] \mid b_t(x,v)\in \partial\X\}
$$
by definition of $I_T^{\X}(x,v)$. Together with \eqref{e:IGHaus} and \eqref{e:formulaforIG} applied to $E=\partial\X$, this concludes the proof of Proposition \ref{p:zelditch}.
\end{proof}

\bibliography{ot.bib}
\bibliographystyle{plain}
\end{document}